\documentclass{amsart}
\usepackage{mathtools}
\usepackage{amscd}
\usepackage{interval}
\usepackage{bm}
\usepackage{float}
\newtheorem{theorem}{Theorem}[section]
\newtheorem{lemma}[theorem]{Lemma}
\newtheorem{prop}[theorem]{Proposition}
\theoremstyle{definition}
\newtheorem{definition}[theorem]{Definition}

\newtheorem{example}[theorem]{Example}

\newtheorem{cor}[theorem]{Corollary}
\newtheorem{conj}[theorem]{Conjecture}
\newtheorem{theoremx}{Theorem}
\theoremstyle{remark}
\newtheorem{remark}[theorem]{Remark}

\numberwithin{equation}{section}

\newcommand{\defeq}{\vcentcolon=}
\newcommand{\cl}{{\rm cl}}
\newcommand{\scl}{{\rm scl}}
\newcommand{\pscl}{{\rm pscl}}
\newcommand{\cone}{{\rm cone}}
\newcommand{\conv}{{\rm conv}}
\newcommand{\supp}{{\rm supp}}
\begin{document}
	\title{Scl in free products}
%	\author{Danny Calegari}
%	\address{Department of Mathematics, University of Chicago, Chicago, Illinois, 60637}
%	\email{dannyc@math.uchicago.edu}
	\author{Lvzhou Chen}
	\address{Department of Mathematics, University of Chicago, Chicago, Illinois, 60637}
	\email{lzchen@math.uchicago.edu}
	
	%    General info
	%\subjclass[2000]{Primary 54C40, 14E20; Secondary 46E25, 20C20}
	
	\date{July 17, 2017. v3.0}
	
	
	%\keywords{Differential geometry, algebraic geometry}
	\begin{abstract}
		We study stable commutator length (scl) in free products via surface maps into a wedge of spaces. We prove that scl is piecewise rational linear if it vanishes on each factor of the free product, generalizing the main result in \cite{DSSS}. We further prove that the property of isometric embedding with respect to scl is preserved under taking free products. The method of proof gives a way to compute scl in free products which lets us generalize and derive in a new way several well-known formulas. Finally we show independently and in a new approach that scl in free products of cyclic groups behaves in a piecewise quasi-rational way when the word is fixed but the orders of factors vary, previously proved by Timothy Susse, settling a conjecture of Alden Walker.
	\end{abstract}
	\maketitle

	\section{Introduction}
	Let $G$ be a group and $g$ be an element of the commutator subgroup $[G,G]$, the \emph{commutator length} of $g$, denoted $\cl(g)$, is the minimal number $n$ such that \linebreak $g=[a_1,b_1][a_2,b_2]\cdots[a_n,b_n]$ for some $a_i,b_i\in G$, and the \textit{stable commutator length} of $g$, denoted $\scl(g)$, is the limit $\lim_{n\to\infty}\cl(g^n)/n$ which always exists by subadditivity.
	
	It is obvious from the definition that $\scl$ has the following basic properties:
	\begin{enumerate}
		\item monotone: for any homomorphism $\phi:G\to H$ and $g\in [G,G]$, we have $\scl_G(g)\ge\scl_H(\phi(g))$;
		\item characteristic: for any $\phi\in{\rm Aut}(G)$ and $g\in [G,G]$, $\scl(g)=\scl(\phi(g))$.
	\end{enumerate}
	
	It follows that the \emph{spectrum}, the set of values that $\scl_G$ takes, is a group invariant. However, $\scl$ is notoriously difficult to compute unless it is known to vanish. Thus many interesting questions about the spectrum are extremely hard to answer.
	
	\subsection{Main results}
	Gromov \cite{Gro} asked whether the spectrum is rational (or perhaps algebraic) when $G$ is finitely presented. A counter-example was found by Zhuang \cite{ZIRRSCL}. On the other hand, Calegari \cite{DFPQL} showed that $\scl$ is rational and can be computed efficiently in a free group by interpreting and studying $\scl$ in terms of surface maps. He later showed in \cite{DSSS} that a modification of the geometric argument proves rationality of scl in free products of abelian groups. We generalize this latter result, substantially weakening the assumption that the factors are abelian.
	\begin{theoremx}[Rationality]\label{rationality}
		Let $G=*_\lambda G_\lambda$ with $\scl_{G_\lambda}\equiv0$ for each $\lambda$, then $\scl$ is piecewise rational linear in $G$.
	\end{theoremx}
	This holds, for example, when all $G_\lambda$ are amenable. See Remark \ref{vanscl} for a list of groups having vanishing scl.
	
	A homomorphism $\phi:G\to H$ for which $\scl_H(\phi(c))=\scl_G(c)$ for all chains $c$ (see Section \ref{sec2}) is said to be \emph{isometric} for scl. Injections admitting a retract are isometric. It is shown by Calegari--Walker \cite{CWIEND} that random homomorphisms between free groups are isometric for scl. In this paper, we show that isometric embeddings (meaning injective and isometric) are preserved under taking free products:
	\begin{theoremx}[Isometric Embedding]
		Let $f_\lambda:H_\lambda\to G_\lambda$ be a family of isometric embeddings, then so is the induced map $f:*_\lambda H_\lambda\to*_\lambda G_\lambda$.
	\end{theoremx}
	
	A spin-off of the techniques used in the proof is a new method to compute scl; we give examples in Section \ref{app}. 
	
	In particular, these techniques give new insights for scl in families. It was proved by Calegari--Walker \cite{CWInthull} that for free products of free abelian groups, certain families of words $w(n)$ (called surgery families) are eventually quasi-rational in $n$. A similar question was studied by Walker: for any fixed rational chain $c$ in $F_n$, and any $\bm{o}=(o_1,o_2,\ldots,o_n)$, with $o_i\ge2$, let $c_{\bm{o}}$ be the image of $w$ under the natural homomorphism $\phi:F_n\to *_i\mathbb{Z}/o_i\mathbb{Z}$, how does $\scl(c_{\bm{o}})$ vary as a function of $\bm{o}$?
	
	It was observed experimentally by Walker \cite{AWscylla} that $\scl(c_{\bm{o}})$ exhibits interesting periodic behavior, and he conjectured that the result is piecewise quasi-linear in $1/o_i$ (see Conjecture \ref{conj}). In Section \ref{wconj} we give a counter-example, but prove a weaker version: $\scl(c_{\bm{o}})$ is piecewise quasi-rational in $\bm{o}$ (see Theorem \ref{weakWconj}). It was pointed out by Timothy Susse that he had proved this weaker version earlier in \cite[Corollary 4.14]{SusseSCL} using a different approach.
	
	It is worth mentioning that the method in this paper can be used to generalize and give a new approach to the spectral gap theorem by Duncan--Howie \cite{DH}, which will be discussed in another paper \cite{SCLGAP}.
	
	\subsection{Contents of paper} We first give basic definitions in Section \ref{sec2}. Then in Section \ref{sec3} we introduce a way, following \cite{DSSS}, to use a finite dimensional polyhedral cone to encode surface maps into a wedge of spaces with given boundary information. The encoding loses information, so in Section \ref{sec4} we study a nonlinear optimization problem on the fibers. This reduces the computation of scl to a lattice point problem, which we solve, deducing Theorem \ref{rationality} and Theorem \ref{isoemb}. When scl vanishes in each factor, the non-linearity comes from \emph{disk vectors}, which become complicated compared to the abelian case discussed in \cite{DSSS}. In Section \ref{app}, we apply our method to give generalizations and new proofs of old results, where we also prove a formula conjectured by Alden Walker in \cite{AWscylla}. Finally in Section \ref{wconj} we give a counter-example to Walker's conjecture and prove a weaker version.
	
	\subsection{Acknowledgment} The author thanks his advisor Danny Calegari for insightful introduction to this topic. The author also thanks Timothy Susse and Alden Walker for useful conversations. Finally the author thanks the referee for nice suggestions.
	
	\section{Background}\label{sec2}
	In this section we give the definitions and basic facts about scl that we will use. All of these can be found in \cite{DSCL}.
	
	\begin{definition}
		Let $S$ be a compact surface. Define $$\chi^-(S)=\sum_i\min(0,\chi(S_i))$$ where $S_i$ are the components of $S$ and $\chi$ is the Euler characteristic. Equivalently, $\chi^-(S)$ is the Euler characteristic of $S$ after removing disk and sphere components.
	\end{definition}
	\begin{definition}
		Let $g_i\in G$ ($1\le i\le k$) such that they sum to $0$ in $H_1(G;\mathbb{R})$. Let $K$ be a $K(G,1)$. For all $i$, let $\gamma_i:S^1\to K$ be a loop representing the conjugacy class of $g_i$ and $L=\sqcup_i S^1$. A compact oriented surface $S$ together with a map $f:S\to K$ is called \textit{admissible} of degree $n(S)\ge1$ if the following diagram commutes
		\[
		\begin{CD}
		\partial S @>{i}>> S\\
		@V{\partial f}VV @V{f}VV\\
		L @>{\sqcup \gamma_i}>> K
		\end{CD}
		\]
		where $i$ is the inclusion map and $\partial f_*[\partial S]=n(S)[L]$.
		
		Define $$\scl(g_1+g_2+\cdots+g_k)=\inf_S\frac{-\chi^-(S)}{2n(S)}$$ over all admissible surfaces.
	\end{definition}
	If $k=1$, the geometric definition agrees with the algebraic one \cite[Proposition 2.10]{DSCL}. We (informally) say a surface map is \textit{efficient} if $-\chi^-(S)/2n(S)$ is close to $\scl(\sum g_i)$.
	\begin{remark}\label{sameori}
		A priori the degrees on different components of $\partial S$ could have opposite signs. Such an admissible surface can be replaced by another one that is at least as efficient as $S$, by taking suitable finite covers and gluing components with opposite orientations together. Thus one may restrict attention to \emph{monotone} admissible surfaces, i.e. $\partial f$ is orientation preserving \cite[Proposition 2.13]{DSCL}.
	\end{remark}
	Recall the complex of real group chains $(C_*(G;\mathbb{R}),\partial)$ whose homology is $H_*(G;\mathbb{R})$, the real group homology of $G$. In the sequel, we write $B_1(G)$ for $B_1(G;\mathbb{R})$, the \linebreak 1-boundaries. scl is defined on integral 1-boundaries, and has a unique continuous linear extension to a pseudo-norm on $B_1(G)$, which vanishes on $$H(G)\defeq\text{span}_\mathbb{R}\left<ng-g^n, g-hgh^{-1}\right>\le B_1(G),$$ thus scl descends to a pseudo-norm on the quotient. See \cite{DSCL} for details.
	\begin{definition}
		Define $B_1^H(G)=B_1(G)/H(G)$. We say scl is \textit{piecewise rational linear} if it is piecewise rational linear on every finite dimensional rational subspace of $B_1^H(G)$. We say a group homomorphism $f:G_1\to G_2$ is an \textit{isometric embedding} if $f$ is injective and the induced map $f:B_1^H(G_1)\to B_1^H(G_2)$ preserves scl, i.e. $\scl_{G_1}(c)=\scl_{G_2}(f(c))$ for all $c\in B_1^H(G_1)$.
	\end{definition}
	The simplest isometric embeddings come from retracts.
	\begin{prop}\label{retract}
		Let $i:H\to G$ and $r:G\to H$ be group homomorphisms such that $r\circ i={\rm id}_H$, then $i$ is an isometric embedding.
	\end{prop}
	This follows immediately from monotonicity of scl.
	\begin{remark}\label{finiteness}
		In particular, the calculation of scl in a free product of infinitely many groups reduces to computations in the free product of finitely many groups.
	\end{remark}
	
	\section{Encoding Surface Maps as Vectors}\label{sec3}
	In this section, we introduce the method from \cite{DSSS} to encode admissible surface maps into a wedge of spaces as vectors in a finite dimensional rational polyhedron.
	
	In the sequel, fix $G=A*B$ to be a free product of two groups $A$ and $B$. Since every finite dimensional rational subspace of $B_1^H(G)$ is a rational subspace of $\left<Z\right>\cap B_1^H(G)$, for some finite subset $Z$ of nontrivial conjugacy classes in $G$, we fix such a $Z$ and study the restriction of scl to $\left<Z\right>\cap B_1^H(G)$. We assume that there are no torsion elements in $Z$ since $ng=g^n=1$ in $B_1^H(G)$ if $g$ is of order $n$.
	
	Let $K_A$ and $K_B$ be a $K(A,1)$ and $K(B,1)$ respectively, then $K=K_A\vee K_B$ is a $K(G,1)$ with wedge point $*$. By choosing appropriate loops to represent elements of $Z$, we get an oriented closed 1-manifold $L$ (one component for each element of $Z$) together with a map $\Gamma:L\to K$ such that for each component $L_i$:
	\begin{enumerate}
		\item either $\Gamma(L_i)$ is disjoint from $*$ and thus contained entirely in $K_A$ or $K_B$ (referred to as \emph{self loops});
		\item or $\Gamma^{-1}(*)\cap L_i$ cuts $L_i$ into finitely many intervals, each mapped alternately to a based loop in one of $K_A$ and $K_B$.
	\end{enumerate}
	Therefore, $L\backslash\Gamma^{-1}(*)$ has finitely many components, each taken to a loop contained in one of $K_A$ and $K_B$ (See Figure \ref{Lex}). Let $T(A)$ and $T(B)$ be the set of components taken to $K_A$ and $K_B$ respectively.
	
	\begin{figure}[H]
		\begin{center}
			\resizebox{222pt}{126pt}{\includegraphics{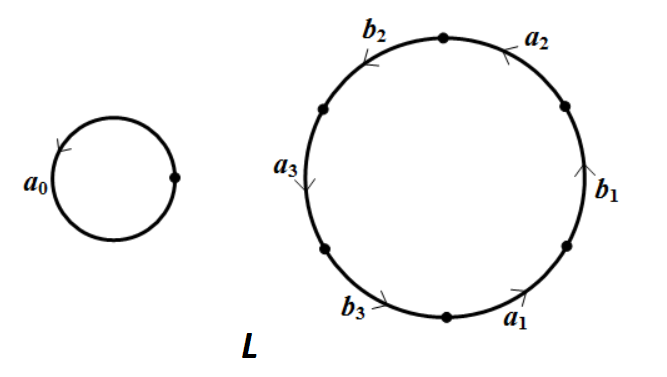}}
			\caption{The 1-manifold $L$ when $Z=\{a_0,a_1b_1a_2b_2a_3b_3\}$; the component on the left is a self-loop.} \label{Lex}
		\end{center}
	\end{figure}
	
	Now for any surface $f:S\to K$ admissible for an integral class in $\left<Z\right>\cap B_1^H(G)$, we may assume up to a homotopy that $\partial f:\partial S\to L$ is a (possibly disconnected) covering map, and assume $f$ is transverse to $*$, i.e. $F\defeq f^{-1}(*)$ is a finite disjoint union of embedded loops and proper arcs. We may also assume (by Remark \ref{sameori}) $\partial f:\partial S\to L$ is orientation preserving.
	
	We can eliminate loops in $F$ by compressing the innermost one each time, which does not increase $-\chi^{-}(S)$. Each proper arc in $F$ is essential in $S$ since $\partial f$ is a covering. So from now on, we assume that $F$ consists of (essential) proper arcs.
	
	Let $S_A$ and $S_B$ be $f^{-1}(K_A)$ and $f^{-1}(K_B)$ respectively, and we focus on $S_A$ in the rest of this section. 
	
	$S_A$ is a surface with corners, and each component of $\partial S_A$ either covers a self loop mapped to $K_A$, or can be decomposed into arcs alternating between components of $F$ and arcs mapped to elements in $T(A)$ (See Figure \ref{SASB}). Note that the corners of $S_A$ are exactly $F\cap\partial S$, thus the orbifold Euler characteristic of $S_A$ $$\chi_o(S_A)\defeq\chi(S_A)-\frac{1}{4}\#(\text{corners})=\chi(S_A)-\frac{1}{2}\#(\text{components of }F),$$ and $S$ can be obtained by gluing $S_A$ and $S_B$ along $F$, hence $$\chi(S)=\chi_o(S_A)+\chi_o(S_B).$$

	\begin{figure}
		\begin{center}
			\resizebox{323pt}{129pt}{\includegraphics{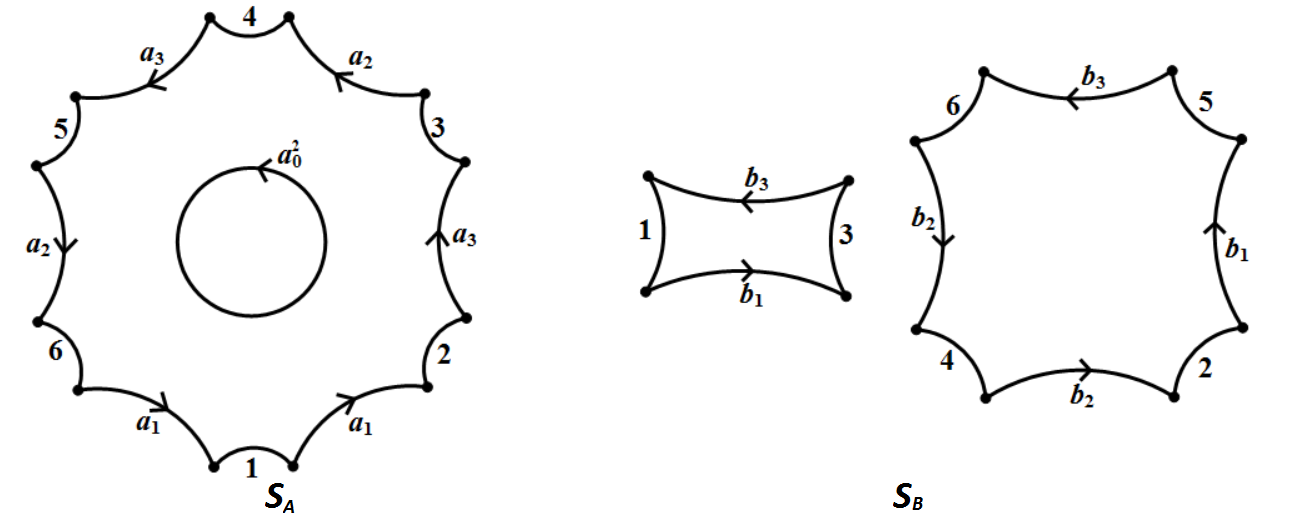}}
			\caption{An example of $S_A$ and $S_B$; components of $F$ are labeled by numbers and arcs with the same label are identified after gluing.} \label{SASB}
		\end{center}
	\end{figure}
	Also note that each component of $F$ with orientation induced from $S_A$ goes from one element of $T(A)$ to another, and thus can be encoded as an ordered pair of these two elements of $T(A)$. Although elements of $T(A)$ corresponding to self loops do not appear in this way, it is convenient to encode a component of $\partial S_A$ that covers a self loop $\tau$ with degree $n$ as $n(\tau,\tau)$, thus we define $$T_2(A)=\{(\tau,\tau')\in T(A)^2|\tau=\tau'\text{ if one of them cooresponds to a self loop}\}.$$
	
	Let $C_1(A)$ and $C_2(A)$ be the $\mathbb{R}$-vector spaces with bases $T(A)$ and $T_2(A)$ respectively, then we can encode the surface $S_A$ as a vector $v(S_A)$ in $C_2(A)$ as follows: each component of $F$ is encoded as an element of $T_2(A)$ described as above, each component of $\partial S_A$ that covers some self loop $\tau$ with degree $n$ is encoded as $n(\tau,\tau)$, and $v(S_A)$ is defined to be the sum of these vectors in $C_2(A)$.
	
	Obviously, $v(S_A)$ is a non-negative integer vector in $C_2(A)$, and it satisfies two more linear constraints. Define a (rational) linear map $\partial:C_2(A)\to C_1(A)$ by $\partial(\tau,\tau')=\tau-\tau'$, then $\partial\circ v(S_A)=0$ since every boundary components of $S_A$ closes up. Define $h:C_2(A)\to H_1(A)\otimes\mathbb{R}$ by $h(\tau,\tau')=\frac{1}{2}(\tau+\tau')$, where $H_1(A)$ is the abelianization of $A$, then $h\circ v(S_A)$ is just the image of $[\partial S_A]$ in $H_1(A;\mathbb{R})$, which is $0$ since it bounds $S_A$.
	
	\begin{definition}\label{VA}
		Let $V_A$ be the convex rational polyhedral cone of non-negative vectors $v\in C_2(A)$ satisfying $\partial(v)=0$ and $h(v)=0$.
	\end{definition}
	
	The discussion above shows that $v(S_A)$ is an integer vector in $V_A$ for any such $S_A$. Conversely, for any integer vector $v\in V_A$, since $\partial(v)=0$, the sum \linebreak $\sum\frac{1}{2}(\tau+\tau')$ actually defines an integral homology class in $H_1(A;\mathbb{Z})$, whose image under $H_1(A;\mathbb{Z})\to H_1(A;\mathbb{Z})\otimes\mathbb{R}\cong H_1(A;\mathbb{R})$ is $h(v)=0$. Hence there is a positive integer $n$ such that the integral homology class given by $nv$ is trivial and thus bounds some (actually many) surface(s). The same thing holds for rational vectors in $V_A$. We summarize this as a lemma for later use.
	
	\begin{lemma}[\cite{DSSS}]\label{Qvec}
		The vector $v(S_A)$ is integral in $V_A$. Conversely, for any rational vector $v\in V_A$, there is an integer $n\ge1$ such that $nv=v(S_A)$ for some $S_A$.
	\end{lemma}
	
	Such an encoding reduces the huge space of admissible surfaces to a finite dimensional space. However, this reduction comes at a cost. There are many different surfaces $S_A$ encoded as the same $v(S_A)$, thus we are led to the following optimization problem: given $v$, a rational vector in $V_A$, what is the infimum of $-\chi_o(S_A)/n(S_A)$ over all surfaces $S_A$ with $v(S_A)=n(S_A)v$ for some $n$? We address this in the next section. 
	
	\section{Nonlinear Optimization}\label{sec4}
	Now we study the optimization problem discussed above. We follow \cite{DSSS}, except that there are significant new issues because the factors are non-abelian. The key observation is Lemma \ref{keylemma}.
	\begin{definition}
		For any rational vector $v\in{V_A}$, define $$\chi_{o,A}(v)=\sup\left\{\left.\frac{\chi_o(S_A)}{n}\right|{v(S_A)=nv}\text{ for some }n\in\mathbb{N}\right\}.$$
	\end{definition}
	As we saw in Section \ref{sec3}, $\chi_o(S_A)=\chi(S_A)-\frac{1}{2}\#(\text{components of }F)$. The number of components of $F$ is a linear function $|v|$ in $v(S_A)$ defined as follows: on the basis, $|(\tau,\tau')|$ is $1$ if $\tau\neq\tau'$, and is $0$ if otherwise; then extend by linearity. Notice that $|v|$ is just the $L^1$ norm if there is no self loop.
	
	Therefore,
	\begin{equation}\label{chi0}
	\chi_{o,A}(v)=-\frac{1}{2}|v|+\sup\left\{\left.\frac{\chi(S_A)}{n}\right|{v(S_A)=nv}\text{ for some }n\in\mathbb{N}\right\}.
	\end{equation}
	Note that the second term is quite similar to the definition of $-2\cdot\scl$, but $S_A$ could have disk components and $\partial S_A$ could be admissible for different chains in $B_1^H(A)$. We first deal with disk components. 
	
	\subsection{Disk Vectors}
	\begin{definition}\label{diskvdef}
		We call $v\in V_A$ a \textit{disk vector} if $v$ encodes some disk. Denote the set of disk vectors by $\mathcal{D}_A$. For $v\in V_A$, we say $v=v'+\sum t_id_i$ is an \textit{admissible expression} if $v'\in V_A$, $t_i\ge0$ and $d_i\in\mathcal{D}_A$. Define $$\kappa_A(v)=\sup\left\{\left.\sum t_i\right|v=v'+\sum t_id_i\text{ is an admissible expression}\right\}.$$
	\end{definition}
	Roughly speaking, $\kappa_A(v)$ is the maximal ``number'' of disk vectors that can be subtracted from $v$. In the case where scl vanishes on $B_1^H(A)$, we have
	\begin{lemma}\label{chi0scl0}
		If scl vanishes on $B_1^H(A)$, then $\chi_{o,A}(v)=-\frac{1}{2}|v|+\kappa_A(v)$ for any rational vector $v\in V_A$.
	\end{lemma}	
	\begin{proof}
		This is equivalent to showing that $$\sup\left\{\left.\frac{\chi(S_A)}{n}\right|v(S_A)=nv\text{ for some }n\in\mathbb{N}\right\}=\kappa_A(v)$$ by equation (\ref{chi0}). Suppose $v(S_A)=nv$ for some $n\in\mathbb{N}$, let $D_1,\dots,D_k$ be the disk components of $S_A$ and $S_A=S_A'\sqcup(\sqcup D_i)$, then $\chi(S_A)=\chi^-(S_A)+k\le k$ and $v=v(S_A')/n+\sum v(D_i)/n$ is an admissible expression. Then $\kappa_A(v)\ge k/n$ and thus $\chi(S_A)/n\le\kappa_A(v)$, this proves the ``$\le$'' direction. 
		
		Conversely, for any given $\epsilon>0$, there is an admissible expression $v=v'+\sum t_id_i$ where $|\kappa_A(v)-\sum t_i|<\epsilon$. We may assume that each $t_i$ is rational, then $v'$ is also rational since $v$ is. Hence there is an integer $n\ge1$ such that each $nt_i$ is an integer and $nv'=v(S_A')$ for some $S_A'$ by Lemma \ref{Qvec}. Now $\partial S_A'$ defines a chain $c$ in $B_1(A)$ where scl vanishes, thus we can find some $S_A''$ such that $\partial S_A''=Nc$, $v(S_A'')=Nv(S_A')=Nnv'$ and $|-\chi^-(S_A'')/N|<\epsilon$. Also find disks $D_i$ such that $v(D_i)=d_i$, and take $Nnt_i$ copies of $D_i$ for each $i$. Finally take $S_A$ to be the disjoint union of all these disks and $S_A''$, then $v(S_A)=Nnv'+Nn\sum t_id_i=Nnv$ and $\chi(S_A)/Nn=\chi(S_A'')/Nn+\sum t_i\ge\chi^-(S_A'')/Nn+(\kappa_A(v)-\epsilon)\ge-\epsilon/n+\kappa_A(v)-\epsilon$. Since $\epsilon$ is arbitrary, this proves the other direction.
	\end{proof}
	This motivates the study of $\kappa_A(v)$ since $|v|$ is already linear on $V_A$. The following lemma is the same as Lemma 3.10 in \cite{DSSS}. The proof is standard, thus we omit it.
	\begin{lemma}\label{minkfunc}
		$\kappa_A$ is a non-negative concave homogeneous function on $V_A$. The subset of $V_A$ on which $\kappa_A=1$ is the boundary of $\conv(\mathcal{D}_A)+V_A$, where ``$+$'' denotes the Minkowski sum.
	\end{lemma}	
	
	\subsection{Key Observation} Now we are coming to the key observation that makes it possible to generalize the result in \cite{DSSS} to our rationality theorem. 
	
	In \cite{DSSS}, essentially using that $A$ is free abelian, $\mathcal{D}_A$ is determined explicitly as integer points lying in some open faces of $V_A$. Then $\conv(\mathcal{D}_A)+V_A$ is shown to be a finitely sided rational convex polyhedron using such an explicit description of $\mathcal{D}_A$, which also produces an effective algorithm to compute scl in that case. However, the set $\mathcal{D}_A$ of disk vectors could be very complicated and hard to determine explicitly in general. 
	
	The following example illustrates how complicated $\mathcal{D}_A$ could be, even when $A$ is the simplest non-abelian group. The study of this example was initiated in an unpublished note by Timothy Susse \cite{Susse}, who did computer experiments and gave conjectural pictures of the result.
	\begin{example}\label{diskex}
		Let $A=\mathcal{H}_3(\mathbb{Z})=\left<x,y,z\ |\ z=[x,y],[x,z]=[y,z]=1\right>$ be the 3-dimensional Heisenberg group, which is $2$-step nilpotent, and $[A,A]=\left<z\right>$. Suppose $T(A)=\{a,b,c\}$ for some $a,b,c\in A\backslash\{id\}$ such that $abc=z^m$ for some $m\in\mathbb{Z}$, which occurs if we consider $g=a\alpha b\beta c\gamma\in [G,G]$ with $G=A*B$ and $\alpha,\beta,\gamma$ in some group $B$.
		
		Let us look at a 2-dimensional subcone of $V_A$ spanned by $P=(a,b)+(b,c)+(c,a)$ and $N=(a,c)+(b,a)+(c,b)$, and find all $(u,v)\in \mathbb{Z}^2_+$ such that $uP+vN$ is a disk vector. For fixed $(u,v)$, the vector $uP+vN$ is a disk vector if and only if there is some cyclic word $w$ in $a,b,c$ such that 
		\begin{enumerate}
			\item $w$ represents $id$ in $A$;
			\item $w$ contains $u$ copies of each of $ab$, $bc$, $ca$ and $v$ copies of $ac$, $cb$, $ba$ as subwords.
		\end{enumerate}
		
		Notice that since $abc=z^m$, $a$ commutes with $bc$ and $cb$, and similarly for $b$ and $c$; we also have $[a,b]=[b,c]=[c,a]=z^n$ for some $n$. Any cyclic word with equal number (say, $k$) of $a,b,c$ in it can be written uniquely as $(abc)^k[a,b]^r$ for some $r\in\mathbb{Z}$ by moving letters around. One can prove by induction (see Appendix) that for fixed $(u,v)$, and the set of cyclic words satisfying the restriction (2) above, the set of $r$ that can appear is
		\begin{equation}\label{Suv}
		S_{u,v}=\left\{\begin{array}{ll}
		\interval{-\frac{v(v+1)}{2}}{\frac{v(v-3)}{2}}\cap\mathbb{Z}& u=v\\
		\interval{-\frac{v(v+1)}{2}}{\frac{v(v-1)}{2}}\cap\mathbb{Z}& u>v\\
		\interval{-\frac{u(u+1)}{2}-v}{\frac{u(u-1)}{2}-v}\cap\mathbb{Z}& u<v.
		\end{array}\right.
		\end{equation}
		
		Therefore, $uP+vN$ is a disk vector if and only if $(abc)^{u+v}[a,b]^r=id$ for some $r\in S_{u,v}$, or equivalently $m(u+v)+nr=0$ has a solution for $r\in S_{u,v}$. For example if $\frac{m}{n}=\frac{1}{2}$, the set of $(u,v)$ for which $uP+vN$ is a disk vector is
		$$\left\{\left.(u,v)\in\mathbb{Z}^2_+\right|1\le v\le u\le v^2\text{ or } u<v\le u^2; \text{ and }u\equiv v \mod 2\right\},$$ which is the set of integer points in the shaded region in Figure \ref{diskv}, bounded by two parabolas, such that the two coordinates have the same parity.
	\end{example}
	
	\begin{figure}
		\begin{center}
			\resizebox{200pt}{200pt}{\includegraphics{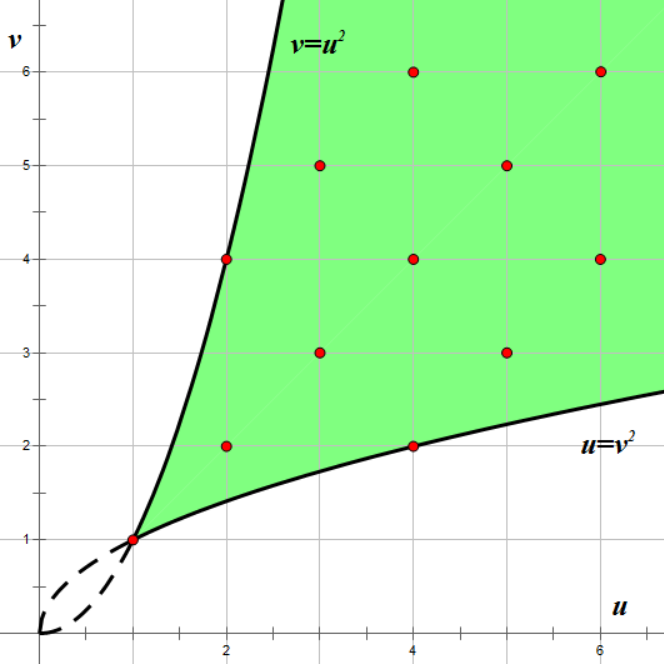}}
			\caption{Disk vectors in a subcone of $V_A$ with $A=\mathcal{H}_3(\mathbb{Z})$.} \label{diskv}
		\end{center}
	\end{figure}
	
	Nevertheless, Corollary \ref{keyob}, a consequence of our key Lemma \ref{keylemma}, shows that $\conv(\mathcal{D}_A)+V_A$ is always a \emph{finitely sided} rational convex cone no matter how complicated $\mathcal{D}_A$ is. The key reason is that an integer point in a rational cone cannot be too close to a given face unless it lies on that face.
	
	We first recall some standard definitions. Let $X=\mathbb{R}^n$.
	\begin{definition}
		A convex \textit{polyhedral cone} in $X$ is a set $C=\{x|f_i(x)\ge0,\forall i\in I\}$, where each $f_i:X\to\mathbb{R}$ is a linear map and $I$ is \emph{finite}. In addition, $C$ is \textit{rational} if each $f_i$ is rational. $C$ is \textit{simplicial} if $f_i$ are linearly independent in $X^*$, or equivalently, $C$ is the convex cone spanned by some linearly independent vectors.
	\end{definition}
	It follows that if $C$ is simplicial and rational, then $C$ is the convex cone spanned by some linearly independent rational vectors. Here is the key observation.
	\begin{lemma}\label{keylemma}
		Let $C$ be a rational polyhedral cone in $X$, and $D$ be a subset of $C\cap(\frac{1}{L}\cdot\mathbb{Z})^n$ for some $L\in\mathbb{Z}_+$. Then there is a finite subset $D'$ of $D$ such that $D+C=D'+C$.
	\end{lemma}
	\begin{proof}
		The claim is trivially true if $D$ is empty. From now on we assume $D$ to be nonempty. We first reduce the problem to the case where $C$ is simplicial. Decompose $C$ into finitely many simplicial rational cones $C_i$, $i=1,\dots,k$. Suppose the claim is true for simplicial rational cones, then let $D_i=D\cap C_i$ and apply the claim to each pair $(C_i,D_i)$, we get some finite sets $D_i'$ such that $D_i'+C_i=D_i+C_i$. Now let $D'=\cup D_i'$. It suffices to show that $D+C\subset D'+C$. Actually, for each point $d+c\in D+C$ where $d\in D$ and $c\in C$, we have $d\in D_i$ for some $i$ since $D=\cup D_i$. Now $D_i\subset D_i+C_i=D_i'+C_i$, thus there exists $d'\in D_i'\subset D'$ and $c'\in C_i\subset C$ such that $d=d'+c'$, thus $d+c=d'+(c'+c)$ lies in $D'+C$.
		
		Therefore, we only need to show the claim for any simplicial rational cone $C$, which can be further reduced as follows to the case where $C$ is the first orthant of $X$. Let $c_i$, $i=1,\dots,k$ be the linearly independent rational vectors that span $C$. Extend this to a rational basis of $X$ and take a linear transformation $f$ of $X$ by sending $c_i$ to $e_i$, where $\{e_i\}_{i=1}^n$ is the standard basis. In terms of matrices (with respect to the basis $e_i$), $f$ is an $n\times n$ matrix with rational entries. Let $N$ be the lcm of the denominators of the entries. Then the image of $(\frac{1}{L}\cdot\mathbb{Z})^n$ under $f$ lies in $(\frac{1}{LN}\cdot\mathbb{Z})^n$, hence $f(D)$ is a subset of $(\frac{1}{LN}\cdot\mathbb{Z})^n$.
		
		Thus we only need to show the claim for $$C=\{x=(x_1,\dots,x_n)|x_i\ge0,i=1,\dots,k;\ x_i=0, i>k\}.$$ Up to applying the map $v\mapsto Lv$ (our statement is irrelevant to the scale), we assume without loss of generality that $L=1$ in the sequel, i.e. $D$ lies in the integer lattice. Now we may ignore $e_i$ for $i>k$. Thus we assume without loss of generality that $C$ is the first orthant of $X$ and proceed by induction on the dimension $n$ (note that $n$ is actually the dimension of $C$, not $X$). See Figure \ref{keyex} for an illustration of our induction in a special case.
		
		The base case $n=1$ is obvious. For the inductive step, fix any $i\in\{1,\dots,n\}$, let $F_i=\{x\in C|x_i=0\}$ be the $i$-th face of $C$ and $p_i$ be the projection from $C$ to $F_i$. Let $D_i=p_i(D)$, which lies in the integer lattice and $F_i$. Thus by induction hypothesis (applied to $(F_i,D_i)$), there is a finite set $D_i'\subset D_i$ such that $D_i'+F_i=D_i+F_i$. For each $x'\in D_i'$, choose some $x''\in D$ such that $p_i(x'')=x'$. Hence there is a finite set $D_i''\subset D$ which projects to $D_i'$ under $p_i$. Note the following simple but crucial fact: for any $x,y\in C$, if $y_i\ge x_i$, then $y$ lies in $x+C$ if and only if $p_i(y)\in p_i(x)+F_i$. Thus if we take $M_i=\max\{x_i|x\in D_i''\}$, then for any point $y$ with $y_i\ge M_i$, we have $y\in D_i''+C$ if and only if $p_i(y)\in D_i'+F_i=D_i+F_i$. Therefore, if $y_i\ge M_i$ and $y\in D$, then $p_i(y)\in D_i\subset D_i+F_i$, which means $y\in D_i''+C$. In other words, if $y\in D\backslash (D_i''+C)$, then $0\le y_i<M_i$.
		
		\begin{figure}
			\begin{center}
				\resizebox{210pt}{215pt}{\includegraphics{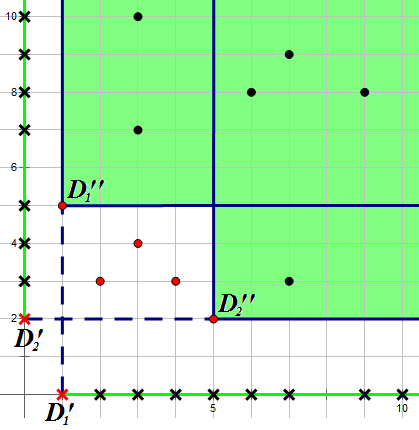}}
				\caption{An illustration of our induction argument in the case $X=\mathbb{R}^2$ and $C=\mathbb{R}_{\ge0}^2$. The dots are points in $D$ and the crosses on the axes are $D_1$ and $D_2$; the red crosses are $D_1'$ and $D_2'$, whose lifts are $D_1''$ and $D_2''$. The shaded region $D''+C$ contains the majority (black dots) of $D$, thus we can take $D'$ to be the red dots.} \label{keyex}
			\end{center}
		\end{figure}
		
		Now let $i$ range from $1$ to $n$ and take $D''=\cup D_i''$. By what we showed above, if $y\in D\backslash(D''+C)=D\backslash[\cup(D_i''+C)]$, then $0\le y_i<M_i$ for each $i$. Hence $D\backslash(D''+C)$ is a bounded subset of $\mathbb{Z}^n$ and therefore finite. Take $D'=D''\cup [D\backslash(D''+C)]$, the claim follows.
	\end{proof}
	\begin{cor}
		Let $X$, $C$ and $D$ be as above, then $\conv(C+D)=\conv(D)+C$ is a closed rational polyhedron.
	\end{cor}
	\begin{proof}
		By the lemma above, $\conv(C+D)$ has finitely many rational (actually integral) vertices (a subset of $D'$), thus $\conv(C+D)$ is a rational polyhedron. $\conv(C+D)=\conv(D)+C$ since the Minkowski sum commutes with taking convex hull and $C$ is convex. It is closed since $\conv(C+D)=\conv(C+D')=\conv(D')+C$ is the Minkowski sum of a compact set and a closed set.
	\end{proof}
	Applying this to $C=V_A$ and $D=\mathcal{D}_A$, we get
	\begin{cor}\label{keyob}
		In any case, $\conv(\mathcal{D}_A)+V_A$ is a (closed) rational polyhedron. Thus $\kappa_A$ is the minimum of finitely many rational linear functions. If scl vanishes on $B_1^H(A)$, then $\chi_{o,A}$ (originally defined on rational vectors in $V_A$) has a (unique) continuous extension $\kappa_A-|\cdot|/2$, which is the minimum of finitely many rational linear functions.
	\end{cor}
	\begin{proof}The first assertion follows immediately from the corollary above. Combining Lemma \ref{minkfunc}, we get the second assertion. Combining Lemma \ref{chi0scl0}, we get the last assertion. 
	\end{proof}
	
	\subsection{Rationality Theorem} Corollary \ref{keyob} generalizes Lemma 3.12 in \cite{DSSS} by weakening the assumption ``(free) abelian'' to ``scl vanishes''. Now following the argument in \cite{DSSS}, we get our first main result:
	\setcounter{theoremx}{0}
	\begin{theoremx}[Rationality]\label{PQL}
		Let $G_\lambda$ ($\lambda\in\Lambda$) be a family of groups where scl vanishes on each $G_\lambda$, then scl is piecewise rational linear on the free product $G=*_\lambda G_\lambda$.
	\end{theoremx}
	Given Corollary \ref{keyob}, the proof is the same as that in \cite{DSSS}. We include it for completeness.
	\begin{proof}
		We first focus on the case of $G=A*B$.
		
		As in Section \ref{sec3}, fix a finite subset $Z$ of nontrivial conjugacy classes in $G$ and define the 1-manifold $L$. Also let $T_2(A)$, $V_A$ and $\chi_{o,A}$ be as above, similarly define these for $B$. Let $Y\subset V_A\times V_B$ be the set of pairs $(v_A,v_B)$ that can be ``glued up'', i.e.: for any $(\tau,\tau')\in T_2(A)$ with $\tau$ not a self-loop (then neither is $\tau'$), we can find a unique $(\sigma,\sigma')\in T_2(B)$ such that $\sigma'$ is the oriented arc in the 1-manifold $L$ following $\tau$, and $\tau'$ follows $\sigma$; we require the $(\tau,\tau')$-coordinate of $v_A$ equals the $(\sigma,\sigma')$-coordinate of $v_B$ for any such $(\tau,\tau')$ and $(\sigma,\sigma')$. Then $Y$ is still a rational cone. Define $\chi_o(v_A,v_B)=\chi_{o,A}(v_A)+\chi_{o,B}(v_B)$ for any $(v_A,v_B)\in Y$. Then by Corollary \ref{keyob}, $\chi_o$ is the minimum of finitely many rational linear functions. Finally define $d:Y\to H_1(L)$ to be the unique rational linear map that $d(y)=\partial f_*(\partial S)$ in $H_1(L)$ whenever $y=(v(S_A),v(S_B))\in Y$ for some surface $S=S_A+S_B$. 
		
		Now for any $l\in H_1(L)$ that corresponds to a chain $z\in\left<Z\right>\cap B_1^H(G)$, let $Y_l=d^{-1}(l)\subset Y$, then we have
		\begin{equation}
			\scl(z)=-\max_{y\in Y_l}\chi_o(y)/2.
		\end{equation}
		Notice that $Y_l$ is a finitely sided convex polyhedron since $Y$ is, and it is rational if $l$ is. Thus scl is computed as a convex programming problem, which can be reduced to a linear programming problem by introducing slack variables because $\chi_o$ is the minimum of finitely many rational linear functions. Then it follows that scl is piecewise rational linear on $\left<Z\right>\cap B_1^H(G)$. Since $Z$ is arbitrary, the conclusion follows.
		
		For the general case, since every rational subspace only involves finitely many factors, it suffices to show that the conclusion holds when $\Lambda$ is finite according to Remark \ref{finiteness}. Now if $\Lambda$ is finite, we can build $K(G,1)$ by gluing up $K(G_\lambda,1)$'s so that no three factors are attached at the same point. This guarantees that surfaces $S_{G_\lambda}$ are glued up in a simple way. Then define $T_2(G_\lambda)$, $V_{G_\lambda}$ and $\chi_{o,G_\lambda}$ as before. Similarly define $Y$ by writing down the suitable gluing condition. Then the same argument above shows that scl is piecewise rational linear.
	\end{proof}
	\begin{remark}\label{vanscl}
		Many groups have vanishing scl. There are three main sources:
		\begin{enumerate}
			\item small groups such as amenable groups, which include finite groups and solvable groups;
			\item irreducible lattices of higher rank Lie groups (see Theorem 5.26 in \cite{DSCL} for a precise statement);
			\item some transformation groups such as Homeo$^+(S^1)$ \cite[Theorem 2.43]{DSCL}, subgroups of PL$^+(I)$ \cite[Theorem A]{DPL}, Homeo$_c(\mathbb{R}^n)$, and Thompson-Stein groups $T_{p,q}$ with $\text{gcd}(p-1,q-1)=1$ (see \cite[Lemma 3.6]{ZIRRSCL} or \cite[Lemma 5.15]{DSCL}).
		\end{enumerate}
	\end{remark}
	\begin{remark}
		The proof actually gives a method to determine scl in free products when scl vanishes on each factor. It produces an algorithm as long as one can determine the vertices of the convex cone $\conv(V_{A_\lambda}+\mathcal{D}_{A_\lambda})$, which seems hard in general since it requires some knowledge of $\mathcal{D}_{A_\lambda}$. The method, however, is still helpful to study scl in families.
	\end{remark}
	\begin{remark}\label{redun}
		When considering $Y_l$, it is redundant in the following sense to impose $h=0$ in the definition of $V_A$ (see Definition \ref{VA}). If we define $V_A'$ to be non-negative vectors $v\in C_2(A)$ satisfying $\partial(v)=0$, then $V_A$ is the sub-polyhedral cone of $V_A'$ on which $h=0$. We can similarly define $Y$ and the linear map $d$ using $V_A'$ instead of $V_A$, denote them by $Y'$ and $d'$. It turns out that if $l\in H_1(L)$ corresponds to homologically trivial chain in $B_1(G)$, then $d'^{-1}(l)\subset Y'$ coincides with $Y_l$.
		%
		%In fact it seems to be more convenient to use $V_A'$ when calculating scl since $V_A'$ is easier to determine, especially when we are not sure about the image of elements of $T_1(A)$ in $H_1(A;\mathbb{R})$. See the detailed calculation in the proof of Proposition \ref{selfprod} to get a feeling. But $V_A$ works more efficiently in algorithm when calculating scl of a family of chains.
	\end{remark}
	\begin{cor}\label{isoembscl0}
		Let $f_\lambda:A_\lambda\to B_\lambda$ be a family of injective group homomorphisms. Suppose scl vanishes on each $A_\lambda$ and $B_\lambda$, then the induced map $f:*_\lambda A_\lambda\to*_\lambda B_\lambda$ is an isometric embedding with respect to scl. More precisely, for any $c\in B_1^H(*_\lambda A_\lambda)$, we have $\scl(c)=\scl(f(c))$.
	\end{cor}
	\begin{proof}
		Again this reduces to the case of finitely many factors. Now run the process above on both sides using $V_{A_\lambda}'$ and $V_{B_\lambda}'$ as in Remark \ref{redun} instead of $V_{A_\lambda}$ and $V_{B_\lambda}$. Then $f_{\lambda}$ identifies $V_{A_\lambda}'$ and $V_{B_\lambda}'$ even if its induced map on homology may not be injective and thus may not identify $V_{A_\lambda}$ and $V_{B_\lambda}$. Injectivity of $f_\lambda$ ensures that $D_{A_\lambda}$ is identified with $D_{B_\lambda}$, thus $\kappa_{A_\lambda}$ is identified with $\kappa_{B_\lambda}$. Then the computation of scl on two sides are results of the same linear programming problem, thus $f$ is isometric for scl.
	\end{proof}
	The condition that scl vanishes on each $A_\lambda$ and $B_\lambda$ ensures that each $f_\lambda$ is isometric. Thus it is natural to ask whether it is enough to get the conclusion only assuming each $f_\lambda$ to be an isometric embedding. Our second main result confirms this. To prove it, we will reveal how scl in free products is determined when scl does not necessarily vanish on each factor.
	
	\subsection{Scl in General Free Products} Now we return to the general case. Similar to the special case discussed above, we need an analog of Lemma \ref{chi0scl0} to reveal the structure of $\chi_{o,A}(v)$. Unlike the case where $\scl$ vanishes on factors, the second term on the right hand side of equation (\ref{chi0}) cannot be computed via $\kappa_{A}$ as in Lemma \ref{chi0} anymore. With notations as before and $\mathcal{D}_A$ defined as in Definition \ref{diskvdef}, we make the following
	\begin{definition}
		For any rational $v\in V_A$, define $$\pscl_A(v)\defeq\inf\left\{\left.\frac{-\chi^-(S_A)}{2n}\right|{v(S_A)=nv}\newline\text{ for some }n\in\mathbb{N}\right\}.$$
		
		Let $\cone_\mathbb{Q}(\mathcal{D}_A)\defeq\{\sum t_id_i|t_i\in\mathbb{Q}, t_i\ge0, d_i\in\mathcal{D}_A\}$. Equivalently, $\cone_\mathbb{Q}(\mathcal{D}_A)$ is the set of rational points in $\cone(\mathcal{D}_A)$. For any $x\in\cone_\mathbb{Q}(\mathcal{D}_A)$, define $$\eta_A(x)\defeq\sup\left\{\left.\sum t_i\right|x=\sum t_id_i \text{ with }t_i\in\mathbb{Q}, t_i\ge0, d_i\in \mathcal{D}_A\right\}.$$
	\end{definition}
	\begin{lemma}\label{chi0gen}
		For any rational $v\in V_A$, we have 
		\begin{eqnarray*}
		\chi_{o,A}(v)&=&-\frac{1}{2}|v|+\sup\left\{\left.-2\pscl_A(v-d)+\eta_A(d)\right|\  d\in\cone_\mathbb{Q}(\mathcal{D}_A),v-d\in V_A\right\}\\
		&\le& -\frac{1}{2}|v|+\kappa_{A}(v).
		\end{eqnarray*}
	\end{lemma}
%	We omit the proof since it is similar to that of Lemma \ref{chi0scl0}.
	\begin{proof}
		We first prove the equality in a similar way as we did for Lemma \ref{chi0scl0}. Let $$L=\sup\left\{\left.\chi(S_A)/n\right|v(S_A)=nv\text{ for some }n\in\mathbb{N}\right\},$$ $$R=\sup\left\{\left.-2\pscl_A(v-d)+\eta_A(d)\right|d\in\cone_\mathbb{Q}(\mathcal{D}_A),v-d\in V_A\right\}.$$ By equation (\ref{chi0}), we just need to show $L=R$.
		
		On the one hand, if $v(S_A)=nv$, let $D_1,\dots,D_k$ be the disk components of $S_A$ and $S_A=S_A'\sqcup(\sqcup D_i)$, then $$\chi(S_A)=\chi^-(S_A)+k\le-2n\pscl_A(v-d)+k\le-2n\pscl_A(v-d)+n\eta_A(d)$$ where $d=\sum v(D_i)/n$, this proves $L\le R$.
		
		On the other hand, for any given $\epsilon>0$, we can find $d\in\cone_\mathbb{Q}(\mathcal{D}_A)$ such that $v'=v-d\in V_A$ and $-2\pscl_A(v')+\eta_A(d)>R-\epsilon$. Then we can write $d=\sum t_id_i$ with $t_i$ nonnegative rational and $d_i\in\mathcal{D}_A$ such that $\sum t_i>\eta_A(d)-\epsilon$. We can also find an integer $n\ge1$ and a surface $S_A'$ such that $v(S_A')=nv'$ and $\frac{\chi^-(S_A')}{2n}>-\pscl_A(v')-\epsilon$. Up to replacing $S_A'$ by a bunch of copies of itself, we may assume that each $nt_i$ is an integer. Now take disks $D_i$ such that $v(D_i)=d_i$. Take $nt_i$ copies of $D_i$ for each $i$ and let $S_A$ be the disjoint union of these disks together with $S_A'$. Then $v(S_A)=nv'+n\sum t_id_i=nv$ and $$\frac{\chi(S_A)}{n}=\frac{\chi(S_A')}{n}+\sum t_i\ge\frac{\chi^-(S_A')}{n}+\sum t_i>-2\pscl_A(v')+\eta_A(d)-3\epsilon>R-4\epsilon.$$ This shows $L\ge R$ and finishes the proof of the equality part.
		
		To show the inequality, suppose $d\in\cone_\mathbb{Q}(\mathcal{D}_A)$ and $v-d\in V_A$. For any $\epsilon>0$, we can write $d=\sum t_i d_i$ such that $\eta_A(d)<\sum t_i+\epsilon$, thus the admissible expression $v=(v-d)+\sum t_i d_i$ shows $$\kappa_A(v)\ge\sum t_i>\eta_A(d)-\epsilon\ge -2\pscl_A(v-d)+\eta_A(d)-\epsilon.$$
		Since $\epsilon$ is arbitrary, the desired inequality holds.
	\end{proof}
	Now we can describe how scl is determined in general free products. Let $G=*_{i}G_i$ be the free product of finitely many groups, consider a finite set $Z$ of conjugacy classes (remove torsion elements) in $G$ and build the 1-manifold $L$ as before, define $V_{G_i}$, $\mathcal{D}_{G_i}$, $\eta_{G_i}$ and $\pscl_{G_i}$ as above. Then define $Y$, as in the proof of Theorem \ref{PQL}, to be the rational polyhedron in $\prod V_{G_i}$ consisting of tuples of vectors from $V_{G_i}$ that can be ``glued up'', and define $\chi_o$ (only for rational points in $Y$) to be the sum of $\chi_{o,G_{i}}$ evaluated on the $i$-th coordinate. Now for any rational $l\in H_1(L)$ that corresponds to a rational chain $z\in\left<Z\right>\cap B_1^H(G)$, let $Y_l=d^{-1}(l)\subset Y$,
	\begin{lemma}\label{detscl}
		 With notations as above, $$\scl(z)=\inf_{{\rm rational}\ y\in Y_l}-\chi_o(y)/2.$$
	\end{lemma}
	The proof follows exactly as before. By ignoring the contribution from $\scl$ in factor groups, we get the following estimate:
	\begin{cor}\label{sclineq} With notations as above,
		$$2\cdot \scl(z)\ge\inf_{y=(v_{G_i})\in Y_l}\sum_i \left[\frac{1}{2}|v_{G_i}|-\kappa_{G_i}(v_{G_i})\right],$$ and equality holds if $\scl_{G_i}\equiv0$ for all $i$.
	\end{cor}
	\begin{proof}
		The inequality follows from Lemma \ref{detscl} and the inequality in Lemma \ref{chi0gen} together with the fact that $\kappa_{G_i}$ has a continuous extension to irrational points (Corollary \ref{keyob}). The equality part is proved in the proof of Theorem \ref{rationality}.
	\end{proof}
	Using Lemma \ref{detscl}, we can generalize Corollary \ref{isoembscl0} to our second main result.
	\setcounter{theoremx}{1}
	\begin{theoremx}[Isometric Embedding]\label{isoemb}
		Let $f_\lambda:H_\lambda\to G_\lambda$ be a family of isometric embeddings with respect to scl, then the induced map $f:*H_\lambda\to*G_\lambda$ is also an isometric embedding.
	\end{theoremx}
	\begin{proof}
		The proof is almost the same as that of Corollary \ref{isoembscl0}. First reduce to finite free products, then apply Lemma \ref{detscl} on both sides accordingly with $V_{H_i}'$ and $V_{G_i}'$ instead as in Remark \ref{redun} to avoid assuming that $f_i$ induces an injective map on homology. Injectivity of $f_i$ ensures that $\mathcal{D}_{H_i}$ can be identified with $\mathcal{D}_{G_i}$, and thus $\eta_{H_{i}}$ is identical to $\eta_{G_{i}}$. Since $f_i$ preserves $\scl$, we see $\pscl_{H_i}$ is identical to $\pscl_{G_i}$. Since $V_{H_i}'$ is identified with $V_{G_i}'$ by $f_i$, $\chi_{o,H_i}$ is the same as $\chi_{o,G_i}$ and thus the computation of scl (for rational chains) on either side is obtained by taking the infimum of the same functions on identical spaces, hence $f$ preserves scl.
	\end{proof}
	\begin{remark}
		Alternately, one can prove this theorem in a more direct way. Here is an outline: 
		
		Reduce to finite free products and then to the case $A*B\to A'*B'$ by induction. It suffices to show $\scl(c)\le\scl'(c')$ for integral homologically trivial chains $c$ where $c'=f(c)$, $\scl=\scl_{A*B}$ and $\scl'=\scl_{A'*B'}$. Take any surface $S'$ mapped into $A'*B'$ that approximates $\scl'(c')$ well, decompose it as in Section \ref{sec3} into pieces $S_{A'}$ and $S_{B'}$. Disk components of $S_{A'}$ ``factor through'' $A$ by injectivity of $A\to A'$, the boundary of the union of other components defines a chain on $A'$ that can be pulled back to a homologically trivial (for the same reason explained in Remark \ref{redun}) chain on $A$ with identical $\scl$ by assumption. Then find a surface that approximates the scl of this chain well, take a finite cover if necessary and take the union with multiple copies of the disk components from $S_A'$, and we obtain a surface $S_A$ with $-\chi_o(S_A)\le m(-\chi_o(S_A')+\epsilon)$ and such that $v(S_A)$ corresponds to $mv(S_A')$, using the notation in Section \ref{sec3}. Do the same thing for $B$ then glue up $S_A$ and $S_B$ (after taking suitable finite covers) to get a surface $S$ mapped to $A*B$ that winds around $c$ and is almost as efficient as $S'$, thus $\scl(c)\le\scl'(c')$.
	\end{remark}
	
	\section{Applications: Generalizations and New Proofs of Old Results}\label{app}
	In this section, we apply the isometric embedding theorem and the computational methods we have developed to get generalizations and new proofs of old results. %The proofs of Proposition \ref{prod} and Proposition \ref{selfprod} are very detailed so they serve as concrete examples to help one understand the discussions above, especially Lemma \ref{detscl}.
	
	We start with a simple corollary of Theorem \ref{isoemb}.
	\begin{cor}\label{steal}
		Let $g_\lambda\in G_\lambda$, $G=*_\lambda G_\lambda$, and $f_\lambda:\left<g_\lambda\right>\to G_\lambda$ be the inclusion. Then the induced map $f:*_\lambda\left<g_\lambda\right>\to G$ is an isometric embedding. In particular, if $g_\lambda\in G_\lambda$ has order $k_\lambda$, then the spectrum of $\scl_G$ contains the spectrum of $\scl$ on $*_\lambda(\mathbb{Z}/k_\lambda\mathbb{Z})$. Here $k_\lambda\ge2$ could be $\infty$, in which case $\mathbb{Z}/k_\lambda\mathbb{Z}$ is $\mathbb{Z}$; and the ``spectrum'' could refer to the values scl takes on $B_1^H$ or the commutator subgroups. 
	\end{cor}
	\begin{proof}
		Simply note that $f_\lambda$ is an isometric embedding even if $\scl_{G_\lambda}(g_\lambda)>0$: because the definition of an isometric embedding is that $f_i$ induces an isometric map $B_1^H(\left<g_\lambda\right>)\to B_1^H(G_\lambda)$, and $B_1^H(\left<g_\lambda\right>)=B_1^H(\mathbb{Z}/k_\lambda\mathbb{Z})=0$ since $\mathbb{Z}/k_\lambda\mathbb{Z}$ is abelian.
	\end{proof}
	\begin{remark}
		Scl in free products of cyclic groups has been studied in \cite{DFPQL}, \cite{DSSS} and \cite{CWIEND}. The program {\rm \ttfamily scallop} \cite{CWscallop} can compute scl on specific chains.
	\end{remark}
	This allows us to generalize results about the scl spectrum in free groups (or free products of cyclic groups) to general free products. For example, we have
	\begin{cor}
		Let $G=*_\lambda G_\lambda$ ($\lambda\in\Lambda$, $|\Lambda|\ge2$) and suppose at least two $G_\lambda$'s contain elements of infinite order. Then the image of $[G,G]$ under scl contains elements congruent to every element of $\mathbb{Q}$ mod $\mathbb{Z}$. Moreover, it contains a well-ordered sequence of values with ordinal type $\omega^\omega$.
	\end{cor}
	\begin{proof}
		This follows from Corollary 3.19 in \cite{CWIEND}, which states that the conclusion is true for nonabelian free groups, and our Corollary \ref{steal}.
	\end{proof}
%	We can also get explicit formulae for scl on certain families of chains. We only list two examples.
%	\begin{cor}
%		Let $G=A*B$, $a\in A$ and $b\in B$ are torsion free, then\\
%		{\rm(1)} $\scl_G(aba^{-1}b^{-1}ab^{-n}a^{-1}b^n)=1-\frac{1}{2n-2}$\\
%		{\rm(2)} Let $w=a^{-\alpha_1-\alpha_2}+b^{-\beta_1-\beta_2}+a^{\alpha_1}b^{\beta_1}+a^{\alpha_2}b^{\beta_2}$ where the $\alpha_i$ are coprime and similarly for the $\beta_i$, and $\alpha_1>\alpha_1+\alpha_2>0>\alpha_2$, $\beta_1>\beta_1+\beta_2>0>\beta_2$, then $$\scl_{G}(w)=\left\{\begin{array}{ll}
%		1-\frac{1}{2}(\frac{1}{\alpha_1}+\frac{\alpha_1-1}{\alpha_1(\beta_1+\beta_2)})& {\rm if}\  \frac{\alpha_1-1}{\alpha_1}\le\frac{\beta_1+\beta_2}{\beta_1}\\
%		1-\frac{1}{2}(\frac{1}{\beta_1}+\frac{\beta_1-1}{\beta_1(\alpha_1+\alpha_2)})& {\rm if}\  \frac{\beta_1-1}{\beta_1}\le\frac{\alpha_1+\alpha_2}{\alpha_1}\\
%		1-\frac{1}{2}(\frac{1}{\alpha_1}+\frac{1}{\beta_1})& {\rm otherwise}
%		\end{array}\right.$$
%	\end{cor}
%	\begin{proof}
%		Apply Corollary \ref{steal}. The results follow from Example 1.1 and Proposition 4.3 of \cite{DSSS} respectively.
%	\end{proof}
%	\begin{remark}
%		If $a$ or $b$ has torsion, we can get similar results by looking at scl on free products of cyclic groups, which is studied in \cite{AWscylla} and can be computed via the -cyclic mode of {\rm \ttfamily scallop}.
%	\end{remark}
	
	Now we give three examples to illustrate how Lemma \ref{detscl} works. We first deduce the following product formula, which was originally stated not quite correctly in \cite{Bavard} (but the proof is still valid for elements of infinite order) and later corrected and proved in \cite{DSCL} for the general case.
	\begin{prop}[Product Formula]\label{prod}
		Let $G=A*B$, and let $a\in A$, $b\in B$ be nontrivial elements. Suppose $a$ and $b$ are of order $n_a$ and $n_b$ (could be $\infty$, in which case $1/\infty=0$ by convention), then $$\scl_G(ab)=\scl_A(a)+\scl_B(b)+\frac{1}{2}(1-\frac{1}{n_a}-\frac{1}{n_b}).$$
	\end{prop}
	\begin{proof}
		If the image of $a$ in $H_1(A;\mathbb{R})$ is not zero, then both sides are $\infty$ by convention, and similarly for $b$. Now we assume this is not the case. Using the notations in previous sections, let $Z=\{ab\}$, then $L$ is just an oriented circle, $V_A=\{t(a,a)|t\ge0\}$, and $\pscl_A(t(a,a))=\scl_A(a)$ by linearity of scl. If $n_a$ is finite, then $\scl_A(a)=0$, $\mathcal{D}_A=\{kn_a(a,a)|k\in\mathbb{Z}_+\}$ and thus $\eta_A(t(a,a))=t/n_a$, $\chi_{o,A}(t(a,a))=-t/2+t/n_a$ by Lemma \ref{chi0gen}. If $n_a=\infty$, then $\mathcal{D}_A=\emptyset$ and $\eta_A(t(a,a))=0$, thus $\chi_{o,A}(t(a,a))=-t/2-2t\scl_A(a)$ by Lemma \ref{chi0gen}. Then $\chi_{o,A}(t(a,a))=-t/2-2t\scl_A(a)+t/n_a$ is valid in both case. Similarly we get $\chi_{o,B}$. Now the ``glue-up'' condition on $V_A\times V_B$ simply requires $s=t$ for $(t(a,a),s(b,b))$, thus $Y=\{(t(a,a),t(b,b))|t\ge0\}$. Then the fundamental class $l\in H_1(L)$ corresponds to the chain $ab$, thus $Y_l$ is a singleton $\{((a,a),(b,b))\}$ and $\chi_o(((a,a),(b,b)))=-1-2\scl_A(a)-2\scl_B(b)+1/n_a+1/n_b$. Therefore by Lemma \ref{detscl}, we get $$\scl_G(ab)=-\chi_o(((a,a),(b,b)))/2=\scl_A(a)+\scl_B(b)+\frac{1}{2}(1-\frac{1}{n_a}-\frac{1}{n_b}).$$
	\end{proof}
	
	The following self-product formula is an analog of the product formula. When $B=\mathbb{Z}$ and $t$ is the generator, it is proved in \cite{DSCL}.
	\begin{prop}[Generalized Self-product Formula]\label{selfprod}
		Let $A$ and $B$ be groups, $x,y\in A$ and $t\in B$ be elements of infinite order. Then $$\scl_{A*B}(xtyt^{-1})=\scl_{A}(x+y)+\frac{1}{2}.$$
	\end{prop}

	\begin{proof}
		Again both sides are $\infty$ if $x+y$ is not $0$ in $H_1(A)$. Thus we assume this is not the case. The result easily follows from the original self-product formula and Theorem \ref{isoemb} by considering ${\rm id}:A\to A$ and the inclusion $i:\left<t\right>\to B$. But we prove it using the computational tool above, which gives a new proof.
		
		We apply Lemma \ref{detscl} to calculate the left hand side. Using notations as before, let $Z=\{xtyt^{-1}\}$, then $L$ is an oriented circle, $C_2(B)$ consists of vectors of the form $v_b=b_{11}(t,t)+b_{12}(t,t^{-1})+b_{21}(t^{-1},t)+b_{22}(t^{-1},t^{-1})$ where we encode the coefficients into a $2\times2$ matrix $b=(b_{ij})$. Then by the definition of $V_B'$ (Remark \ref{redun}), $\partial(v_b)=0$ requires the sum of entries in the $i$-th row equals that of those in the $i$-th column for all $i$, which is $b_{12}=b_{21}$ in this case. Similarly $V_A'=\{u_a|a_{12}=a_{21}, a_{ij}\ge0\}$, where $u_a=a_{11}(x,x)+a_{12}(x,y)+a_{21}(y,x)+a_{22}(y,y)$. If $(u_a,v_b)\in Y'$, the glue-up condition requires $a_{11}=b_{12}$, $a_{12}=b_{11}$, $a_{21}=b_{22}$ and $a_{22}=b_{21}$. In other words, $b$ is the matrix we get by interchanging the columns of $a$. Finally $(u_a,v_b)\in Y_l$ requires in addition that each row of $b$ (and $a$) sums up to $1$. Together with $\partial(v_b)=0$, this implies that $b_{12}=b_{21}$.
		
		In summary, $Y_l=\{(u_{M(\alpha)},v_{M(1-\alpha)})\mid\alpha\in[0,1]\}$ where
		\[
			M(x)=\begin{pmatrix}x& 1-x\\1-x&x\end{pmatrix}
		\]
		
		Now $\pscl_B(v_{M(1-\alpha)})=0$ since all $t$ and $t^{-1}$ will cancel. Since $t$ has infinite order, $\mathcal{D}_B+V_B'=\{(t,t^{-1})+(t^{-1},t)\}+V_B'$ which implies $\eta_B(\beta[(t,t^{-1})+(t^{-1},t)])=\beta$ and further $\eta_B(v_{M(1-\alpha)})=\alpha$. Thus $\chi_{o,B}(v_{M(1-\alpha)})=-1+\alpha$ by Lemma \ref{chi0gen}. For $\chi_{o,A}$, it is more straight forward to use equation (\ref{chi0}). Thus we have $$\chi_o(u_{M(\alpha)},v_{M(1-\alpha)})=-2+\alpha+\sup\{\chi(S_A)/n\mid v(S_A)=n\cdot u_{M(\alpha)}\}.$$
		
		Therefore by Lemma \ref{detscl}, we only need to show $$1+2\scl_A(x+y)=\inf_{\alpha\in[0,1]\cap\mathbb{Q}} \{2-\alpha+\inf\{-\chi(S_A)/n\mid v(S_A)=n\cdot u_{M(\alpha)}\}\}.$$ 
		
		Let $\alpha=1$, then $u_{M(1)}=(x,x)+(y,y)$ and thus $$\inf\{-\chi(S_A)/n\mid v(S_A)=n\cdot u_{M(1)}\}=2\scl_A(x+y)$$ since $S_A$ has no disk components because $x$ and $y$ have infinite order. This gives the ``$\ge$'' direction. 
		
		Conversely, we just need to show that $2\scl_A(x+y)\le 1-\alpha-\chi(S_A)/n$ always holds. In fact, since $v(S_A)=nu_{M(\alpha)}=n(1-\alpha)[(x,y)+(y,x)]+n\alpha[(x,x)+(y,y)]$, there are $2n(1-\alpha)$ edges on the boundary of $S_A$, mapped to the wedge point $*$, that sit in between an $x$ and a $y$. Half of these edges are from $x$ to $y$ (referred to as a $(x,y)$-edge) and the other half are from $y$ to $x$ (referred to as a $(y,x)$-edge). Whenever we have a $(x,y)$-edge and a $(y,x)$-edge that lie on the same boundary component, we glue a rectangle to the surface with one edge glued to the $(x,y)$-edge and its opposite edge glued to the $(y,x)$-edge, and let $f$ map the rectangle to the wedge point. Such a surgery increases $-\chi$ by $1$. Repeating the process we get a new surface $S_A'$ such that $-\chi(S_A')=-\chi(S_A)+n(1-\alpha)$ and each boundary component either winds around $x$ several times or around $y$. This implies that $S_A'$ has no disk components since $x$ and $y$ have infinite order and $\partial S_A'$ winds around each of $x$ and $y$ $n$-times in total. Thus $1-\alpha-\chi(S_A)/n=-\chi^-(S_A')/n\ge\scl_A(x+y)$. This completes the proof.
	\end{proof}
	
	Finally we prove the following formula which was conjectured for free products of cyclic groups and proved for $G=\mathbb{Z}*(\mathbb{Z}/m\mathbb{Z})$ by Alden Walker in \cite{AWscylla}. It was pointed out by Timothy Susse that in the case of free product of cyclic groups, this is equivalent to Proposition 4.1 he proved in \cite{SusseSCL} by considering certain amalgams of abelian groups.
	\begin{prop}\label{abAB}
		Let $G=A*B$, $a\in A\backslash\{id\}$ and $b\in B\backslash\{id\}$, then $$\scl_G([a,b])=\frac{1}{2}-\frac{1}{k}$$ where $2\le k\le+\infty$ is the minimum of the orders of $a$ and $b$.
	\end{prop}
	\begin{proof}
		By Theorem \ref{isoemb}, we may assume $A=\left<a\right>$ and $B=\left<b\right>$. Let $k_a$ and $k_b$ be the orders of $a$ and $b$ respectively.
		
		Similar to the proof of Proposition \ref{selfprod}, we have $$Y_l=\{(u_{M(\alpha)},v_{M(1-\alpha)})\mid \alpha\in[0,1]\}$$ where 
		\[
		M(x)=\begin{pmatrix}x& 1-x\\1-x&x\end{pmatrix}\quad \]
		and we get
		$$\kappa_A(u_{M(\alpha)})=1-\alpha+\frac{2\alpha}{k_a}, \quad \kappa_B(v_{M(1-\alpha)})=\alpha+\frac{2(1-\alpha)}{k_b}.$$
		Therefore $$-\chi_o(u_{M(\alpha)},v_{M(1-\alpha)})=2-1-\frac{2\alpha}{k_a}-\frac{2(1-\alpha)}{k_b}=1-\frac{2\alpha}{k_a}-\frac{2(1-\alpha)}{k_b},$$
		which has maximum $1-2/k$ for $\alpha\in[0,1]$, thus $$\scl_G([a,b])=\frac{1}{2}-\frac{1}{k}.$$
	\end{proof}
%	\begin{remark}
%		Other formulas conjectured in \cite{AWscylla} can be proved by the same method but the computations are more complicated.
%	\end{remark}
	\section{Walker's Conjecture}\label{wconj}
	Fix a rational chain $c$ in $F_n$, for any $\bm{o}=(o_1,o_2,\ldots,o_n)$, with $o_i\ge2$, let $c_{\bm{o}}$ be the image of $c$ under the natural homomorphism $\phi:F_n\to *_i\mathbb{Z}/o_i\mathbb{Z}$. It is natural to ask: how does $\scl(c_{\bm{o}})$ depend on $\bm{o}$?
	
	Based on computer experiments, Alden Walker conjectured in \cite{AWscylla} the following formulas:
	\[
	\begin{array}{lll}
		c=aba^{-2}b^{-2}+ab& \scl(c_{\bm{o}})=2/3-\{2/3,1/2\}/\min(o_1,o_2)&\text{if }\min(o_1,o_2)\ge2\\
		c=aba^{-3}b^{-3}& \scl(c_{\bm{o}})=3/4-1/o_1-1/o_2&\text{if }\min(o_1,o_2)\ge7\\
		c=a^2ba^{-1}b^{-1}a^{-2}bab^{-1}& \scl(c_{\bm{o}})=1/2-\{2,1\}/o_1&\text{if }\min(o_1,o_2)\ge3\\
		c=aba^2b^2a^3b^3a^{-5}b^{-5}& \scl(c_{\bm{o}})=1-\frac{1}{2o_1}-\frac{1}{2o_2}&\text{if }\min(o_1,o_2)\ge6\\
	\end{array}
	\]
	where $o_1$ and $o_2$ are the orders of $a$ and $b$ respectively, and brackets indicate that the coefficients depend on congruence classes: for example, $\{2,1\}/o_1$ means $2/o_1$ if $o_1\equiv 0$ mod $2$ and $1/o_1$ if $o_1\equiv 1$ mod $2$.
	
	Motivated by this, Walker proposed the following conjecture:
	\begin{conj}[Walker \cite{AWscylla}]\label{conj}
		For any fixed chain $c$ in $F_n$, $\scl(c_{\bm{o}})$ is piecewise quasilinear in $1/o_i$, i.e. there are some $p\in\mathbb{Z}_+$, and a finite partition of $\mathbb{Z}_{\ge 2}^n$, such that on each piece, fixing any congruence class of each $o_i$ mod $p$, $\scl(c_{\bm{o}})$ is linear in $1/o_i$.
	\end{conj}
	
	Computer experiments suggest that this conjecture is false.
	
	\begin{example}
		The following formula holds experimentally for $n=2$ and $c=aba^{-2}b^{-2}a^2b^2a^{-1}b^{-1}$ with $o_2/2>o_1>10$:
		\[
		\scl(c_{\bm{o}})=\left\{\begin{array}{ll}
		1-\frac{3(o_1-1)}{o_1(o_1+1)}& o_1\equiv 1, 3, 5 \mod 6\\
		1-\frac{3}{o_1}& o_1\equiv 0 \mod 6\\
		1-\frac{15}{5o_1+8}& o_1\equiv 2 \mod 6\\
		1-\frac{3}{o_1+2}& o_1\equiv 4 \mod 6
		\end{array}\right.
		\] 
		This is verified by the computer program {\ttfamily scallop} \cite{CWscallop} for $o_2=100$ and $10<o_1<50$. We see from this example that the denominator could be higher degree polynomial in $o_1$, and even when it is linear in $o_1$, it could be inhomogeneous.
	\end{example}
	
	To seriously disprove the conjecture, it suffices to verify a special case of the formula above:
	\begin{prop}
		For $n=2$ and $c=aba^{-2}b^{-2}a^2b^2a^{-1}b^{-1}$, there exist constants $r,s\ge1$ such that when $o_1=6K+3$, $o_2=6L+3$ with $L\ge rK$ and $K\ge s$, we have 
		$$\scl(c_{\bm{o}})=1-\frac{3(o_1-1)}{o_1(o_1+1)}.$$
	\end{prop}
	We prove the ``$\le$'' direction and give an outline of the proof for the other direction.
	\begin{proof}
		Follow the notations in Section \ref{sec4} and apply our method to $G=A*B$ with $A=\mathbb{Z}/o_1\mathbb{Z}$ and $B=\mathbb{Z}/o_2\mathbb{Z}$. Then $T(A)=\{a,a^{-2},a^2,a^{-1}\}$. Let
		\begin{eqnarray*}
		v_A&=&\frac{1}{3K+2}[(a,a^{-1})+(a^{-1},a)]+\frac{1}{3K+2}[(a^2,a^{-2})+(a^{-2},a^2)]\\
		&+&\frac{K}{(2K+1)(3K+2)}[(6K+3)(a,a)]+\frac{1}{(6K+3)(3K+2)}[(6K+3)(a^{-1},a^{-1})]\\
		&+&\frac{1}{3K+2}[(a,a^2)+3K(a^2,a^2)+(a^2,a)]\\
		&+&\frac{1}{3K(3K+2)}[(a^{-1},a^{-2})+3K(a^{-2},a^{-2})+(a^{-2},a^{-1})]\\
		&+&\frac{9K^2-1}{K(3K+2)(6K+3)}[(2K+1)(a^{-2},a^{-1})+(2K+1)(a^{-2},a^{-1})],
		\end{eqnarray*}
		where each bracket is a disk vector. In particular, we know $v_A\in V_A$ and
		\begin{eqnarray*}
		\kappa_{A}(v_A)&\ge&\frac{3}{3K+2}+\frac{K}{(2K+1)(3K+2)}+\frac{1}{(6K+3)(3K+2)}\\
		&+&\frac{1}{3K(3K+2)}+\frac{9K^2-1}{K(3K+2)(6K+3)}\\
		&=&\frac{30K+12}{(3K+2)(6K+3)}.
		\end{eqnarray*}
		
		Similarly, let
		\begin{eqnarray*}
		v_B&=&\frac{3K}{3K+2}[(b,b^{-1})+(b^{-1},b)]+\frac{3K}{3K+2}[(b^2,b^{-2})+(b^{-2},b^2)]\\
		&+&\frac{1}{3K+2}[(b^2,b^{-1})+(b^{-1},b^{-1})+(b^{-1},b^2)]\\
		&+&\frac{1}{3K+2}[(b,b^{-2})+(b^{-2},b^{-2})+(b^{-2},b)+(b,b^2)+(b^2,b)]\in V_B,
		\end{eqnarray*}
		where each bracket is a disk vector and
		$$\kappa_B(v_B)\ge 2-\frac{2}{3K+2}.$$
		
		One can check that $v_A$ and $v_B$ satisfy the gluing condition and $(v_A,v_B)\in Y_l$ where $l$ is the fundamental class of the loop representing the chain $c$. Therefore by Corollary \ref{sclineq} (the equality part), for any $K,L\ge0$, we have  $$\scl(c_{\bm{o}})\le \frac{1}{2}[2-\kappa_{A}(v_A)]+\frac{1}{2}[2-\kappa_B(v_B)]\le1-\frac{18K+6}{(6K+3)(6K+4)}=1-\frac{3(o_1-1)}{o_1(o_1+1)}.$$
		
		For the other direction, we only need to show that $(v_A,v_B)$ constructed above achieves the maximum of the optimization problem
		$$(P_0)\quad \text{maximize: }\kappa_A(u)+\kappa_B(w)\quad \text{subject to: } (u,w)\in Y_l,$$ and that the estimates for $\kappa_{A}(v_A)$ and $\kappa_B(v_B)$ above are sharp. The key idea is to use duality of linear programming. Here is an outline:
		
		\begin{enumerate}
			\item We linearize this optimization problem $(P_0)$ in a way similar to \cite{CFL}. On the ``$A$'' side, consider the directed graph (as in \cite{DSSS}) with vertex set $T(A)$ and directed edge set $T(A)^2$. Let $SL_A$ be the set of directed simple (i.e. visiting each vertex at most once) loops. Each directed loop cyclically visiting vertices $a_1,\ldots,a_n$ corresponds to a vector $\sum_{i=1}^{n}(a_i,a_{i+1})$ in $V_A'$. Then disk vectors can be written (not uniquely) as linear combinations of simple loops with non-negative integral coefficients. One can enumerate disk vectors that are \emph{extremal}, i.e. cannot be written as a convex combination of other disk vectors plus a non-negative linear combination of simple loops. It turns out that there are finitely many ($169$) extremal disk vectors and each depends linearly on $K$, which is compatible with Lemma \ref{diskfamily} below. Denote the set of extremal disk vectors by $ED_A$. Obtain $SL_B$ and $ED_B$ on the ``$B$'' side simply by substituting $a$ and $K$ by $b$ and $L$ respectively since two sides have the same structure. Then $(P_0)$ can be linearized as:
			$$(P)\quad \text{maximize: }f^{T}x \quad \text{subject to: } Cx=b \text{ and } x\ge\bm{0} \text{ (entrywise)},$$
			where $x=(x_i)$ and $f=(f_i)$ are indexed by $SL_A\sqcup ED_A\sqcup SL_B\sqcup ED_B$, $f_i=1$ if $i\in ED_A\sqcup ED_B$, $f_i=0$ otherwise, and the constraint $Cx=b$ corresponds to gluing and normalization conditions.
			
			\item The way we decompose $v_A,v_B$ into disk vectors gives rise to a feasible solution $x_0$ to $(P)$. Our goal is to show that $x_0$ achieves the maximum. To accomplish this, it suffices to find $y_0$ such that $$C^{T}y_0\ge f \text{ (entrywise)} \text{ and }x_0^{T}C^{T}y_0=x_0^{T}f.$$ This proves the maximality because $$f^{T}x=x^{T}f\le x^{T}C^{T}y_0=b^{T}y_0=x_0^{T}C^{T}y_0=x_0^{T}f=f^{T}x_0.$$
			One such $y_0$ (in an explicit formula involving $K$ and $L$) can be guessed out via results found by computers for small values of $K$ and $L$ ($v_A$ and $v_B$ are also found in this way). The constants $r$ and $s$ come into the statement because the author only checked $C^{T}y_0\ge f$ when $L/K$ and $K$ are large enough.
		\end{enumerate}
	We omit the details since it is tedious and takes too much space to enumerate the extremal disk vectors and check $C^{T}y_0\ge f$.
	\end{proof}

	Nevertheless, a weaker version of Walker's conjecture is true:
	
	\begin{theorem}\label{weakWconj}
		For any fixed rational chain $c$ in $F_n$, $\scl(c_{\bm{o}})$ is piecewise quasi-rational in $\bm{o}$, i.e. there are some $p\in\mathbb{Z}_+$, and a finite partition of $\mathbb{Z}_{\ge 2}^n$, such that on each piece, fixing any congruence class of each $o_i$ mod $p$, $\scl(c_{\bm{o}})$ is in $\mathbb{Q}(\bm{o})$.
	\end{theorem}
	
	Timothy Susse (\cite{SusseSCL}, Corollary 4.14) proved the same result by considering a fixed chain in a family of amalgamations of free abelian groups, whose projection to the free product of cyclic groups preserves scl. Our proof is independent and new.
	
	We focus on a single factor $A=\mathbb{Z}/k\mathbb{Z}$. Using notations as in Section \ref{sec4}, the key is to show that the vertices of $\conv(\mathcal{D}_A+V_A)$ behave nicely as $k$ varies in congruence classes (see Lemma \ref{diskfamily}). Since $H_1(A;\mathbb{R})$=0, we have $h=0$, thus $V_A$ consists of non-negative vectors in $C_2(A)\cap\{\partial=0\}$ and does not depend on $k$. However $\mathcal{D}_A$ typically depends on $k$, and we denote it by $D_k$ to emphasize the dependence.
	
	We first describe $D_k$. For simplicity, we assume $n=2$, and $c=a_1b_1\ldots a_mb_m$ is a single word, but the proof of lemmas are the same for the general case. Consider the directed graph $X(A)$ with vertex set $T(A)$ and edge set $T_2(A)$. Then each $v\in V_A$ defines non-negative weights on the directed edges, and its support, $\supp(v)$, is the subgraph of $X(A)$ consisting of edges with positive weights.
	
	Let $a,b$ be the generators of $F_2$ giving the free product structure. Then each $a_i=a^{t_i}$ for some $t_i\in\mathbb{Z}\backslash\{0\}$. Let $\tilde{h}: C_2(A)\to\mathbb{R}$ be the linear map such that for any $(a_i,a_j)\in T_2(A)$, $\tilde{h}(a_i,a_j)=(t_i+t_j)/2$.
	
	Then it is easy to see that $D_k$ is the set of integer vectors $v$ in $V_A$ such that $\tilde{h}(v)\in k\mathbb{Z}$ and $\supp(v)$ is connected and nonempty (see \cite{DSSS} for details).
	
	We can decompose $V_A$ into finitely many simplicial rational open faces, i.e. each is of the form $$\left\{\left.\sum_{i=1}^{d} t_iv_i\right|t_i>0 \right\},$$ for some $d\ge1$ and a set of linearly independent rational vectors $v_i$. Moreover, each simplicial rational cone can be decomposed into finitely many \emph{unimodular} cones, i.e. where we can take the set of $v_i$ to be unimodular, by Barvinok's theorem (\cite{BIntpt}, Chapter 16). So we first prove the following key lemma leading to Lemma \ref{diskfamily} and Theorem \ref{weakWconj}. %Thus it suffices to prove Lemma \ref{diskfamily} for a single open unimodular cone. This case follows from the second part of Lemma \ref{std} by considering $f=\tilde{h}$ and $-\tilde{h}$. 
	
	\begin{lemma}\label{std}
		Let $V=\mathbb{R}^d_{>0}$ and $f(x)=\sum a_ix_i$ ($a_i\in\mathbb{Q}$) be a rational linear function. Let $V_k=f^{-1}(k)\cap V$ and $E_k$ be the set of integer points in $V_k$. Then there are $M, p\in\mathbb{Z}_+$ such that:
		\begin{enumerate}
			\item for each congruence class mod $p$, there are finitely many points $v_j(k)\in V$ that depend linearly on $k$ such that $\conv(E_k+V_k)=\conv(\{v_j(k)\}+V_k)$ for any $k>M$ in this given congruence class;
			\item for each congruence class mod $p$, there is a finite set $F_k$ of points depending linearly on $k$, such that $$\conv(\cup_{t\in\mathbb{Z}_+} E_{tk}+V)=\conv(F_k+V)$$ for any $k>M$ in the given congruence class mod $p$. More precisely, we can take $F_k=\cup_{t=1}^{p}\{v_j(tk)\}$ for any $k>M$.
		\end{enumerate}
	\end{lemma}
	Lemma \ref{std} is similar in spirit to the following special case of the main theorem of \cite{CWInthull}, which we will use in our proof.
	
	\begin{lemma}[Calegari--Walker \cite{CWInthull}]\label{QIQ}
		Let $\{\xi_i(k)\}$ be a finite set of points depending linearly on $k$, and then there are $M,p\in\mathbb{Z}_+$ such that the vertices (finitely many) of the integer hull of $\conv(\xi_i(k))$ depend linearly on $k>M$ in each congruence class mod $p$. 
	\end{lemma}
	
	\begin{proof}[Proof of Lemma \ref{std}] We first prove (2) modulo (1). Notice that each $v_j(k)$ depends linearly on $k$ and stays in $V$, thus if $k'>k>M$ and $k'\equiv k$ mod $p$, then $v_j(k')\in v_j(k)+V$. Also notice that any $tk$ is congruent to some $t_0k$ with $1\le t_0\le p$. Hence the vertices of $\conv(\cup_{t\in\mathbb{Z}_+} E_{tk}+V)$ are contained in $F_k=\cup_{t=1}^{p}\{v_j(tk)\}$, and the assertion holds.
		
	Now we prove (1). We may assume all $a_i$'s are non-zero, otherwise we can do a dimension reduction. Let $P$ and $N$ be the set of indices such that $a_i$ is positive or negative respectively. If $P=\emptyset$, then $E_k=\emptyset$ and the problem is trivial, so we also assume $P\neq\emptyset$ in the sequel. Let $\{e_i\}$ be the standard basis of $\mathbb{R}^d$. For any $i\in P$, let $\xi_i(k)=ke_i/a_i\in V_k$.
		
	If $N=\emptyset$, i.e all $a_i>0$, then $V_k$ is the interior of the simplex with vertices $\{\xi_i(k)\}$ and its set of integer points $E_k$ coincides with that of the polyhedron $$\Delta_k\defeq \{(x_1,x_2,\ldots,x_d)\mid x_i\ge1 \}\cap\conv\{\xi_i(k)\}.$$ When $k>\sum_i a_i$, $\Delta_k$ is the (compact) simplex with vertices $$\left\{\sum_{j\neq i}e_j+\frac{k-\sum_{j\neq i}a_j}{a_i}e_i\right\}_{i=1}^d$$ depending linearly on $k$, so our assertion follows from Lemma \ref{QIQ}.
		
	Now also suppose $N\neq\emptyset$, then $V_k=\conv\{\xi_i\mid i\in P\}+V_0$. We first deal with integer points in each $\xi_i+V_0$. Pick $p$ such that $p/a_i\in\mathbb{Z}$, then for $k=tp+k_0$ with $0\le k_0\le p-1$ fixed, $\xi_i+V_0$ is $t(p/a_i)\cdot e_i+C$ where $C=(k_0/a_i)\cdot e_i+V_0$ is a translate of $V_0$ which does not depend on $k$, therefore in this congruence class, the integer hull of $\xi_i+V_0$ is just that of $C$ translated by $t(p/a_i)e_i$, a vector depending linearly on $k$.
		
	If $x\in V_k$ is not contained any $\xi_i+V_0$ (this does not happen for $|P|=1$, so we assume $|P|\ge2$ below), then for each $i\in P$, $x_i\le k/a_i$, hence $x$ lies in $$C_k\defeq V_k\cap\left(\bigcap_{i\in P}\{x\mid x_i\le k/a_i\}\right).$$ The set of integer points in $C_k$ coincides with that in $$Q_k\defeq\{x\in V_k|x_i\ge1,\forall i,\text{ and } x_i\le k/a_i,\forall i\in P\}.$$ $Q_k$ is compact since $1\le x_i\le k/a_i$ for any $i\in P$ and $x\in V_k$ implies $$x_j\le k(|P|-1)/(-a_j)\quad \forall j\in N.$$ To see the vertices of $Q_k$, consider its decomposition into the following level sets: $$Q_k^{(t)}\defeq\left\{x\in Q_k\left|\sum_{i\in P}a_i x_i=t\right.\right\}.$$
	When $k(|P|-1)\ge\sum_{j\in N}(-a_j)$ and $k\ge \sum_{i=1}^d a_i$, the set $Q_k^{(t)}$ is non-empty if and only if $k-\sum_{j\in N}a_j\le t\le k|P|$. For such $t$, one can see that $Q_k^{(t)}$ is the product of
	$$\left\{(x_i)_{i\in P}\left|1\le x_i\le k/a_i,\forall i\in P, \sum_{i\in P} a_ix_i=t\right.\right\}\ \text{(combinatorially a level set of a cube)}$$
	and $$\left\{(x_j)_{j\in N}\left|x_j\ge1,\forall j\in N, \sum_{j\in N}(-a_jx_j)=t-k\right.\right\}\quad \text{(a simplex)}.$$
	From this, we can see that the vertices of $Q_k$ are of the form 
	\begin{eqnarray*}
	&\left.\left.\phantom{\frac{k}{a_i}}\right\{x\ \right|&x_i=1\text{ or }\frac{k}{a_i},\forall i\in P,\text{and }\\
	&&\left.\exists l\in N, s.t.\ x_j=1,\forall j\in N-\{l\}, x_l=\frac{\sum_{i\neq l}a_ix_i-k}{-a_l}\ge1\right\}
	\end{eqnarray*}
	or 
	\begin{eqnarray*}
	&\left.\left.\phantom{\frac{k}{a_i}}\right\{x\ \right|& x_j=1, \forall j\in N,\text{ and }\\
	&&\left.\exists l\in P\ s.t.\ x_i=1\text{ or }\frac{k}{a_i}, \forall i\in P-\{l\}, x_l=\frac{k-\sum_{i\neq l}a_ix_i}{a_l}\in\left[1,\frac{k}{a_l}\right]\right\},
	\end{eqnarray*}
	each depending linearly on $k$, so Lemma \ref{QIQ} applies. Since $V_k$ is the union of $\xi_i+V_0$ ($i\in P$) and $C_k$, and the integer hull of each part has vertices depending linearly on $k\gg1$ in a congruence class, so our assertion follows.
	\end{proof}
	
	Now we can prove
	\begin{lemma}\label{diskfamily}
		There are $M, p\in\mathbb{Z}_+$ such that for each congruence class
		mod $p$, we can find finitely many points $v_j\in V_A$, each depending \emph{linearly} on $k$, such that $\conv(D_k+V_A)=\conv(\{v_j\}+V_A)$ for any $k>M$ in this given congruence class.
	\end{lemma}
	\begin{proof}
		According to the discussion ahead of Lemma 6.4, we can express $V_A$ as the union of top-dimensional faces (denote them by $V(i)$) of finitely many simplicial \emph{unimodular} (Barvinok's theorem) rational cones, and the intersection of $D_k$ with each $V(i)$ is either empty (when the support is disconnected) or exactly the integer points in $V(i)\cap\tilde{h}^{-1}(k\mathbb{Z})$. Apply Lemma \ref{std} to each $V(i)$ with $f=\tilde{h}$ and $f=-\tilde{h}$ respectively (together with the set $V(i)\cap\tilde{h}^{-1}(0)$ that does not depend on $k$), we see that there are $M,p\in\mathbb{Z}_+$ such that for each congruence class mod $p$, we can find finitely many points $v_j(i)$ such that $\conv(D_k\cap V(i)+V(i))=\conv(\{v_j(i)\}+V(i))$ for any $k>M$ in this given congruence class. This completes the proof by taking the union since there are only finitely many $i$'s.
	\end{proof}
	
	Now we prove Theorem \ref{weakWconj}.
	
	\begin{proof}[Proof of Theorem \ref{weakWconj}]
		It follows from Lemma \ref{diskfamily} that for $k\gg1$ in a fixed congruence class mod $p$, $\kappa_A$ is the minimum of finitely many linear functions each having coefficients in $\mathbb{Q}(k)$. Here $A$ can be any factor group and $k$ is the corresponding $o_i$. Therefore, if we fix the congruence classes of $o_i\gg1$, combining the proof of Theorem \ref{rationality}, $\scl(c_{\bm{o}})$ is determined by minimizing, on a fixed compact convex set $C$, the maximum of finitely many linear functions $f_j$ each having coefficients in $\mathbb{Q}(\bm{o})$. Thus we can find a finite polyhedral decomposition of $C$ with vertices having $\mathbb{Q}(\bm{o})$ coordinates, and $\max_j\{f_j\}$ linear on each piece. It follows that $\scl$ is the minimum of finitely many functions in $\mathbb{Q}(\bm{o})$, i.e. the values of $\max_j\{f_j\}$ on these finitely many vertices, hence $\scl$ is piecewise $\mathbb{Q}(\bm{o})$.
	\end{proof}
		
	\section{Appendix}
	Here we give a proof of equation (\ref{Suv}). For convenience, we use $\#_{s}(w)$ to denote the number of subwords $s$ inside $w$. Let $W_{u,v}$ be the set of cyclic words $w$ in $a,b,c$ such that $w$ contains $u$ copies of each of $ab$, $bc$, $ca$ and $v$ copies of $ac$, $cb$, $ba$ as subwords.
		
	For each $w\in W_{u,v}$, let $f(w)$ be the unique integer such that $w$ can be written as $(abc)^k[a,b]^{f(w)}$ by moving letters around and using $[a,b]=[b,c]=[c,a]$. In Example \ref{diskex}, we defined $S_{u,v}$ to be the image of $W_{u,v}$ under $f$.
	
	In order to prove the equation inductively, we first introduce a way to reduce the computation of $S_{u,v}$ to that of smaller indices. 
	
	For each $w\in W_{u,v}$, the letter $a$ appears $u+v$ times in $w$. For convenience, we make the following:
	\begin{definition}
		An $a$-connecting subword of $w$ is the subword between two consecutive $a$'s in $w$.
	\end{definition}
	For example, if $abcba$ is a subword of $w$, then $bcb$ is an $a$-connecting subword of $w$. We classify all $a$-connecting subwords and divide them into three categories:
	\begin{enumerate}
		\item degree $1$: $b(cb)^k c$ with $k\ge0$;
		\item degree $0$: $b(cb)^k$ or $c(bc)^k$ with $k\ge0$;
		\item degree $-1$: $c(bc)^k b$ with $k\ge0$. 
	\end{enumerate} 
	
	\begin{lemma}\label{reduction}
		If there are two degree $1$ $a$-connecting subwords in $w\in W_{u,v}$, then we can find $w_1,w_2\in W_{u-1,v}$ such that $$f(w_1)\le f(w)\le f(w_2).$$
	\end{lemma}
	\begin{proof}
		Up to a cyclic permutation, $w=ab(cb)^k cRab(cb)^l cT$ where $R$ and $T$ are empty words or subwords starting with $a$, and $k,l\ge0$. Recall that $abc$ is in the center and $a$ commutes with $bc$, thus $ab(cb)^l c=a^{-l}(abc)^{l+1}$ and
		$$[ab(cb)^l c,R]=[a^{-l},R]=[a,b]^{l(\#_c(R)-\#_b(R))}.$$
		
		Thus 
		$$w=ab(cb)^kcab(cb)^lcRT\cdot [ab(cb)^l c,R]^{-1}=ab(cb)^kcab(cb)^lcRT\cdot [a,b]^{l(\#_b(R)-\#_c(R))}.$$
		Notice that $ab(cb)^k\underline{cab}(cb)^lcRT$ is still in $W_{u,v}$, and removing the underlined $cab$ which is followed by $c$, we will get a word $w_1=ab(cb)^{k+l}cRT\in W_{u-1,v}$ and $f(w_1)=f(w)+l(\#_b(R)-\#_c(R))$.
		
		Similarly $$w=Rab(cb)^k\underline{cab}(cb)^lcT\cdot [a,b]^{k(\#_c(R)-\#_b(R))},$$
		so $w_2=Rab(cb)^{k+l}cT\in W_{u-1,v}$ and $f(w_2)=f(w)-k(\#_b(R)-\#_c(R))$.
		
		Hence if $\#_b(R)-\#_c(R)\le0$, we are done; otherwise switch $w_1$ and $w_2$.
	\end{proof}
	
	\begin{proof}[Proof of equation (\ref{Suv})]
		Notice the following symmetry: Reading a word $w\in W_{u,v}$ in reverse order gives a word $r(w)\in W_{v,u}$ and $f(r(w))=-f(w)-(u+v)$, thus $$S_{u,v}=-S_{v,u}-u-v.$$
		
		According to Lemma \ref{reduction}, if $w\in W_{u,v}$ has two degree $1$ $a$-connecting subwords in $w$, then $f(w)\in S_{u-1,v}$ assuming that $S_{u-1,v}$ consists of integers in an interval. Similarly by the symmetry above, if $w\in W_{u,v}$ has two degree $-1$ $a$-connecting subwords in $w$, then $f(w)\in S_{u,v-1}$ assuming that $S_{u,v-1}$ consists of integers in an interval.
		
		First assume we have proved the equation for $u=v+1$. We induct over $u-v$ to show that $S_{u,v}=S_{v+1,v}$ whenever $u\ge v+1$. Suppose $w\in W_{u,v}$ with $u>v+1$, notice that an $a$-connecting subword $w_0$ has degree $d$ if and only if 
		$$[\#_{ab}(aw_0a)+\#_{ca}(aw_0a)]-[\#_{ac}(aw_0a)+\#_{ba}(aw_0a)]=2d.$$
		
		Also notice that if we sum the left hand side of the equation above, over all $a$-connecting subwords, we will get $2(u-v)\ge4$. Hence we conclude that there exist two degree $1$ $a$-connecting subwords in $w$, thus $f(w)\in S_{u-1,v}$ since $S_{u-1,v}$ consists of integers in an interval by induction hypothesis. This shows that $S_{u,v}\subset S_{u-1,v}$, but the other inclusion is obvious: adding a copy of $abc$ ahead of a letter $a$ in $w\in S_{u-1,v}$ will result in a new word $w'\in S_{u,v}$ with $f(w')=f(w)$.
		
		Therefore, using the symmetry, we only need to prove the equation for $S_{u,v}$ with $|u-v|\le1$, and we induct on $u+v$. The base cases are easy to check. We now show $$S_{v+1,v}=\left[-\frac{v(v+1)}{2},\frac{v(v-1)}{2}\right]\cap\mathbb{Z}$$
		assuming (\ref{Suv}) holds for all $S_{u',v'}$ with $u'+v'<2v+1$ and $|u'-v'|\le1$.
		
		Consider the following family of words in $W_{v+1,v}$: $$w_k=a(bc)^{k+1}ac(bc)^{v-k}(ac)^{v-1}(ab)^v,\ 0\le k\le v.$$
		A direct computation shows that $f(w_k)=v(v-3)/2+k$. This together with arguments before shows that $S_{v,v}\cup [v(v-3)/2,v(v-1)/2]\subset S_{v+1,v}$, hence by induction hypothesis, $$\left[-\frac{v(v+1)}{2},\frac{v(v-1)}{2}\right]\cap\mathbb{Z}\subset S_{v+1,v}.$$
		
		So we only need to show $$\max(S_{v+1,v})\le v(v-1)/2 \text{ and }\min(S_{v+1,v})\ge -v(v+1)/2.$$ Suppose $w\in W_{v+1,v}$ achieves $\max(S_{v+1,v})\ge f(w_v)=v(v-1)/2$, we see that
		\begin{enumerate}
			\item $w$ does not contain subwords $abca$, $bcab$ or $cabc$, otherwise $f(w)\in S_{v,v}$, which has maximum $v(v-3)/2<v(v-1)/2$ by induction;
			\item $w$ does not contain subwords $acba$, $bacb$ or $cbac$, otherwise $f(w)\in S_{v+1,v-1}$, and $S_{v+1,v-1}=S_{v,v-1}$ has maximum $(v-1)(v-2)/2-v<v(v-1)/2$ by induction;
			\item $w$ does not contain the subword $abaca$, since it can be replaced by $acaba$ to get a new word $w'\in W_{v+1,v}$ with $f(w')>f(w)$;
			\item only one $a$-connecting subword in $w$ has degree $1$, others have degree $0$, otherwise there will be at least two degree $1$ subwords (since the sum of degrees is $1$), which implies (by Lemma \ref{reduction} and induction hypothesis) $f(w)\le\max(S_{v,v})$ contradicting maximality.
 		\end{enumerate}
 		
 		Therefore $w$ must be of the form (up to replacing it by another that also achieves the max) $$w=abc(bc)^kac(bc)^{p_1}a\cdots ac(bc)^{p_s}acabab(cb)^{q_1}ab(cb)^{q_t},$$ where $s,t\ge0$, $k\ge0$ and $p_i,q_j\ge0$. Since $w\in W_{v+1,v}$, we see $s=t=v-1$ and $k+\sum p_i+\sum q_j=v$. A direct computation shows $$w=(abc)(ab)^v(bc)^v(ab)^v[a,b]^e,$$ where $e=vk+\sum (v-i)p_i+\sum jq_j$. Maximizing $f(w)$ is the same as maximizing $e$, which requires $p_i=q_j=0$ and $k=v$. Therefore $w=w_v$ as we constructed and $\max(S_{v+1,v})\le f(w_v)=v(v-1)/2$.
 		
 		Similarly we can show $\min(S_{v+1,v})\ge-v(v+1)/2$, hence (\ref{Suv}) holds for $S_{v+1,v}$, and for $S_{v,v+1}$ by symmetry. 
 		The inductive step for $S_{v,v}$ is completely similar, so we omit it. This completes the proof. 
	\end{proof}
	\bibliographystyle{amsplain}

\end{document}